\documentclass{article}
\usepackage[utf8]{inputenc}
\usepackage{amsmath}
\usepackage{amssymb}
\usepackage{amsthm}
\usepackage{color}
\usepackage{soul}
\usepackage{verbatim}
\usepackage{url}
\usepackage{geometry}
\usepackage{graphicx}
\usepackage{asymptote}
\usepackage{hyperref}
\usepackage[left]{showlabels}
\usepackage{bbm}
\usepackage[shortlabels]{enumitem}

\newtheorem*{rep@theorem}{\rep@title}
\newcommand{\newreptheorem}[2]{%
\newenvironment{rep#1}[1]{%
 \def\rep@title{#2 \ref{##1}}%
 \begin{rep@theorem}}%
 {\end{rep@theorem}}}
\makeatother

\newcommand{\cS}{\mathcal{S}}

\newcommand{\cC}{\mathcal{C}}

\newcommand{\R}{\mathbb{R}}
\newcommand{\C}{\mathbb{C}}

\newcommand{\bI}{\overline{I}}
\newcommand{\bJ}{\overline{J}}
\newcommand{\one}{\mathbbm{1}}
\newcommand{\fkS}{\mathfrak{S}}

\newcommand{\cm}{\operatorname{CM}}

\DeclareMathOperator{\ncm}{NCM}
\newcommand{\tw}{\tilde{w}}

\theoremstyle{plain}
\newtheorem*{Thm}{Theorem}

\newtheorem{thm}{Theorem}[section]
\newtheorem{prop}[thm]{Proposition}
\newtheorem{cor}[thm]{Corollary}
\newtheorem{lemma}[thm]{Lemma}

\newreptheorem{lemma}{Lemma}

\theoremstyle{definition}
\newtheorem{definition}{Definition}
\theoremstyle{remark}

\newtheorem*{remark}{Remark}
\newtheorem*{example}{Example}

\newtheorem{question}[thm]{Question}

\DeclareMathOperator{\imm}{Imm}

\DeclareMathOperator{\sgn}{sgn}

\title{$\%$-Immanants and Temperley-Lieb Immanants}
\author{Frank Lu, Kevin Ren, Dawei Shen, Siki Wang}
\date{March 2023}

\begin{document}

\maketitle
\begin{abstract}

In this paper, we investigate the relationship between Temperley-Lieb immanants, which were introduced by Rhoades and Skandera, and \%-immanants, an immanant based on a concept introduced by Chepuri and Sherman-Bennett. Our main result is a classification of when a Temperley-Lieb immanant can be written as a linear combination of \%-immanants. This result uses a formula by Rhoades and Skandera to compute Temperley-Lieb immanants in terms of complementary minors. Using this formula, we also derive an explicit expression for the coefficients of a Temperley-Lieb immanant coming from a $321$-, $1324$-avoiding permutation $w$ containing the pattern $2143,$ which we use to derive our main result.
\end{abstract}


\tableofcontents
\section{Introduction}
Immanants are functions defined on square matrices which are generalizations of the determinant. For a function $f:\mathfrak S_n\to\mathbb C$, we define the immanant associated to $f$, $\imm_f:M_{n\times n}\to \mathbb C$, as
\[\imm_f(X)= \sum\limits_{\sigma \in \mathfrak{S}_n} f(\sigma)x_{1, \sigma(1)}x_{2, \sigma(2)} \cdots x_{n, \sigma(n)},\] where the matrix $X$ has entries $x_{i, j}.$ We may also view this immanant as a polynomial in $\mathbb{C}[x_{i, j}],$ where $1 \leq i, j \leq n.$

We are interested in the following two families of immanants.


First, we have Temperley-Lieb immanants, introduced in \cite{RS} by Rhoades and Skandera. These immanants are indexed by $321$-avoiding permutations, and their coefficients are derived from coefficients of expansions of a product within the Temperley-Lieb algebra. They are also a special case of another type of immanant, called Kazhdan-Lusztig immanants, as proven in \cite{HarderRS}. Unlike Temperley-Lieb immanants, Kazhdan-Lusztig immanants are a lot more difficult to understand, as their coefficients are related to Kazhdan-Lusztig polynomials, which in turn are related to each other by a recursive relation. 

We then introduce a new class of immanants which we call \%-immanants. These immanants are based on the notion of the determinant formulas from \cite{CSB}, and are significantly easier to compute than Temperley-Lieb immanants.
\begin{definition}[\%-immanants]
Suppose $\lambda / \mu$ is a skew tableau. The \%-immanant associated to this permutation is defined by $$\imm^{\%}_{\lambda/\mu}(X) = \sum_{\sigma \in A} \sgn(\sigma) x_{1, \sigma(1)}x_{2, \sigma(2)} \cdots x_{n, \sigma(n)},$$ where $\sigma \in A$ iff for all $i$, $(i,\sigma(i)) \in \lambda /\mu$. This is related to a `Skew Ferrers Matrix' of \cite{MR2353118}.
\end{definition}


Recently, using the connections between Kazhdan-Lusztig immanants and Schubert varieties, Chepuri and Sherman-Bennett \cite{CSB} proved that for a $2143$ and $1324$-avoiding permutation $w$, the corresponding Kazhdan-Lusztig immanant $\imm_w(M)$ is the same as some $\%$-immanant (up to signs). In particular, the same result holds for Temperley-Lieb immanants. Motivated by this result, our paper seeks to answer the following question.

\begin{question}
Which Temperley-Lieb immanants are linear combinations of \%-immanants?
\end{question}

The advantage of specializing to the case of Temperley-Lieb immanants is that determining the coefficients of these immanants, while still difficult, can be done non-recursively, suggesting that such a question is more tenable for Temperley-Lieb immanants than Kazhdan-Lusztig immanants in general. In this paper, we provide a complete answer to the above question. Furthermore, our proofs will be purely combinatorial, based on the relationship between Temperley-Lieb immanants, non-crossing matchings, and colorings. Our main result is the following.

\begin{Thm}
Let $w$ be a $321$-avoiding permutation. The following statements are equivalent:
\begin{enumerate}
    \item The Temperley-Lieb immanant $\imm_w$ is a linear combination of \%-immanants,
    
    \item The signed Temperley-Lieb immanant $\sgn(w) \imm_w$ is a sum of at most two \%-immanants,
    
    \item The permutation $w$ avoids the patterns $1324, 24153, 31524, 231564$, and $312645,$ in addition to avoiding $321.$
\end{enumerate}
\end{Thm}


\par The plan of the paper is as follows. In section 2, we go through the important concepts and introduce some of the important objects that we'll be studying, and review some preliminary results about them. Following that, we cover \%-immanants in section 3, where we also present a result describing which immanants can be written as linear combinations of \%-immanants. 
\par In section 4, we prove that the Temperley-Lieb immanant of a $321$-avoiding permutation $w$ is a \%-immanant (in fact, the \%-immanant associated with $w$) if and only if $w$ avoids the patterns $2143$ and $1324$ using combinatorial methods. Although both directions can be proved relatively easily with known results (e.g. \cite{billey2003maximal}, \cite{CSB}), the setup here will evoke the style of proof we will use in the rest of the paper.
\par We conclude in section 5 by proving that the Temperley-Lieb immanant of a $321$-avoiding permutation is a sum of two \%-immanants, and more generally a linear combination of \%-immanants, if and only if it avoids the patterns $1324, 24153, 31524, 231564, 312645.$ In the course of proving this result, we will also arrive at a relatively simple combinatorial formula for the coefficients of a Temperley-Lieb immanant associated with a $321$- and $1324$-avoiding permutation $w$ that contains the pattern $2143.$ We note that this formula, along with the work in section 4, gives us the coefficients of the Temperley-Lieb immanant of any permutation $w$ that avoids $321$ and $1324.$
\par Throughout the paper, we will leave the proofs of technical lemmas at the end of each section, so that the reader can first focus on the bigger picture. An interested reader is welcomed to go through the proofs of the lemmas. 
\par We have an upcoming paper that extends our results in this paper to general Kazhdan-Lusztig immanants, i.e., we do not require the associated permutation to be $321$-avoiding.

\section{Preliminaries}\label{sec:prelims}

\par In this section, we go through the main preliminary concepts that we extensively use throughout this paper. We start with some discussion of permutations and Bruhat order, and then introduce Temperley-Lieb immanants. We finally discuss complementary minors and their relation to Temperley-Lieb immanants.
\subsection{Permutations}
We begin this section by introducing some notation which will help us talk about permutations.

\par For a finite set $I$, let $|A|$ denote the cardinality of $A$. If furthermore we have a set of indices $I \subset [n] := \{1, 2, \ldots, n\},$ let $\overline{I} = [n] - I.$ 
\par Given integers $a, b,$ let $[a, b] = \{a, a+1, \ldots, b\}$ if $a \leq b$ and the empty set otherwise. Furthermore, we let $[a, b: c, d]$ be $[a, b] \cup [c, d],$ and similarly we let $[a_1, a_2:a_3, a_4: a_5, a_6: \ldots: a_{2n-1}, a_{2n}]$ be the union of the $n$ intervals $[a_{2i - 1}, a_{2i}]$ for $i \in [n].$ Let $s(A)$ be the sum of the elements of a set $A \subset \mathbb{R}.$

\par Let $\fkS_n$ denote the group of permutations of length $n$. A permutation $u \in \fkS_n$ can be viewed as a map $[n] \to [n]$. Thus, given a subset $S \subset [n]$, we can define $u(S)$ to be the image of $S$, or $\{ u(x) : x \in S \}$. 

\par Given $v, w \in \fkS_n,$ let $v \cdot w$ to denote the product of $v, w$ in $\fkS_n,$ given by $v\cdot w(i) = v(w(i)).$ We will sometimes drop the $\cdot$ when writing the product for readability. We furthermore let $(i, j)$ denote the transposition swapping $i$ and $j$. It will be clear by context whether $(i, j)$ denotes a transposition or an ordered pair. 
\par We also have the \textbf{longest word} $w_0 \in \fkS_n,$ which is the permutation with one-line notation $n (n-1) (n-2) \cdots 1.$ We will specify which permutation group we are considering the longest word in when it is unclear from context.
\par A large part of this paper is devoted to extracting substructures from a permutation. We define two mechanisms for doing so. First, we have a notion of a block structure for a permutation, which we define as follows.
\begin{definition}A \textbf{block} is a string of consecutive ascending integers. Given two disjoint blocks $[a]$ and $[b],$ we say that $[a] < [b]$ if the largest element in $[a]$ is smaller than the smallest element of $[b].$
\par From here, given a permutation $w \in \mathfrak{S}_n$ and $v \in \mathfrak{S}_m,$ where $m \leq n,$ we say that $w$ has \textbf{block structure} $[v(1)][v(2)]\ldots [v(m)]$ if the one-line notation for $w$ consists of $m$ blocks, where $[1] < [2] < \ldots < [m].$
\end{definition}
For instance, the permutation with one-line notation $56123784$ has block structure $[3][1][4][2].$ Note that a permutation can have multiple block structures (e.g. the same permutation has block structure $[4][1][2][5][3]$) but a unique block structure with the fewest number of blocks.

\par We will sometimes want to be explicit about the values that are present in a given block. To do this, in analogy with the notation $[a, b]$ for the set $\{a, a+1, \ldots, b\},$ we let $(a..b)$ denote the sequence $(a, a+1, \ldots, b),$ and we let $(a_1..a_2: a_3..a_4: \ldots: a_{2n-1}..a_{2n})$ denote the sequence $$(a_1, a_1 + 1, a_1 + 2, \ldots, a_2, a_3, a_3 + 1, \ldots, a_4, \ldots, a_{2n-1}, a_{2n-1} + 1, \ldots, a_{2n}).$$ If $a > b,$ then we let $(a..b)$ denote the empty sequence. The most common place we will use this notation is for writing the one-line notation of a permutation. We can think of the one line notation as a sequence of elements, where the $i$-th element in the sequence is $w(i).$ For instance, we can write the one line notation $56123784$ using our notation as $(5..6:1..3:7..8:4..4).$ 
\par For brevity, given a sequence $(a_1..a_2, a_3..a_4, \ldots, a_{2n-1}..a_{2n})$ and a function $f$ whose values are well-defined on all of the elements in the sequence, we will denote the sequence $$(f(a_1), f(a_1 + 1), \ldots, f(a_2), f(a_3), f(a_3 + 1), \ldots, f(a_4), \ldots, f(a_{2n-1}), f(a_{2n-1} + 1), \ldots, f(a_{2n}))$$ as $f(a_1..a_2:a_3..a_4:\ldots:a_{2n-1}..a_{2n}).$ The most common situation that we will run into during this paper is the situation where the $a_i$ all lie in $[n],$ and $f \in \fkS_n.$

\par Our other mechanism for obtaining substructures comes from the idea of restriction, in analogy with the notion of restricting the domain of functions. Note that the following definition is the same operation as the flattening operation from \cite[\S 3]{billey2003maximal}. 
\begin{definition}
Given a permutation $w \in \fkS_n$ and a set of indices $I \subset [n]$, \textbf{the restricted permutation} $w|_I$ is the permutation on $\fkS_{|I|}$ defined by the following condition: for $1 \le i \le |I|$, $w$ maps the $i$-th smallest element of $I$ to the $w|_I (i)$-th smallest element of $w(I)$.
\end{definition}

\par The restriction of a permutation allows us to naturally define pattern avoidance.

\begin{definition}
    Given a permutation $w \in \fkS_n$ and $v \in \fkS_m$, with $m \le n$, we say that $w$ \textbf{avoids} the pattern $v$ if there does not exist a subset $I \subset [n]$ with $|I| = m$ such that $w|_I = v$.
\end{definition}

For example, the permutation $31524$ avoids $321$ but does not avoid $123$.

Pattern avoidance (and by contrapositive, containing a pattern) is closed under taking inverses and by application of the longest word. This is the subject of the following two lemmas.
\begin{lemma}\label{lem: restriction inverses}
Let $w \in \fkS_n.$ Then, if $v = w|_I \in \fkS_m$ for some subset $I \subset [1, n]$ (which in particular requires $|I| = m$), then $v^{-1} = (w^{-1})|_{w(I)}.$ In particular, $w$ avoids the pattern $v$ if and only if $w^{-1}$ avoids the pattern $v^{-1}$.
\end{lemma}

\begin{lemma}\label{lem: restriction flips}
Let $w \in \fkS_n.$ Then, if $v = w|_I \in \fkS_m$ for some subset $I \subset [1, n],$ and $w_0 \in \fkS_n, w_0' \in \fkS_m$ are the longest words in their respective permutation groups, then $w_0'v = (w_0w)|_I$ and $vw_0' = (ww_0)|_{w_0(I)}$.
\end{lemma}

Applying the above lemma twice yields the following corollary.
\begin{cor}\label{cor: w_0 conjugation pattern avoidance}
Let $w \in \fkS_n.$ Then, if $v = w|_I \in \fkS_m$ for some subset $I \subset [1, n],$ and $w_0 \in \fkS_n, w_0' \in \fkS_m$ are the longest words in their respective permutation groups, then $w_0'vw_0' = (w_0ww_0)|_{w_0(I)}.$ In particular, $w$ avoids the pattern $v$ if and only if $w_0 w w_0$ avoids the pattern $w_0' v w_0'$.
\end{cor}
We note that for this paper, we will only need Lemma \ref{lem: restriction inverses} and Corollary \ref{cor: w_0 conjugation pattern avoidance}, and we will be mainly applying these to a handful of values of $v$ (namely, $321$ and the five patterns that appear in the statement of the main theorem). 
\subsection{Bruhat order}
\par We now define the Bruhat order on $\fkS_n$. See \cite{bjorner2006combinatorics} for a detailed reference.

\begin{definition}\label{defn: Bruhat Order}
A \textbf{reduced word} for an element $w \in \fkS_n$ is a decomposition of $w$ into the simple reflections $s_1, s_2, \ldots s_{n-1}$ of $\fkS_n,$ where $s_i$ is the transposition $(i, i+1),$ that is minimal in length among all such decompositions.
\par The length of a permutation $\ell(u)$ is the length of a reduced word for $u$.
\par The \textbf{Bruhat order} on $\fkS_n$ is defined as follows. We say that $u \le v$ if one of the following three equivalent definitions are satisfied:
\begin{enumerate}
    \item \label{bruhat1} Some reduced word for $v$ contains a subword equal to $u$.

    \item \label{bruhat2} There exists a sequence $u = u_1, u_2, \cdots, u_k = v$ such that for each $1 \le i \le k-1$, we can express $u_i^{-1} u_{i+1}$ as a transposition $(a_i, b_i)$ and we also have $\ell(u_{i+1}) = \ell(u_i) + 1$.

    \item \label{bruhat3} For all $1 \le i, j \le n$, we have
$|u([1,i]) \cap [1,j]| \ge |v([1,i]) \cap [1,j]|$.
\end{enumerate}
Using equivalent definition 2, it follows that the Bruhat order is graded by the length function: if $u \le v$, then $\ell(u) \le \ell(v)$. The minimal element of the Bruhat order is the identity permutation, and the maximal element is the longest word $w_0$ in $\fkS_n,$ which has length $\ell(w_0) = \frac{n(n-1)}{2}.$
\end{definition}

We remark that a fourth alternate definition for Bruhat order is given as Theorem 2.1.5 in \cite{bjorner2006combinatorics}. 
\par For example, note that $1423 < 2431,$ since we have the sequence $1423, 2413, 2431,$ with each adjacent pair satisfying the condition stated in definition \ref{bruhat2} above. 
\par In a few places, we wish to compare two permutations that differ on a prescribed set of indices. The following lemma makes this possible.

\begin{lemma}\label{lem:restriction}
Let $v, w \in \fkS_n$. If $I \supset \{ i \mid v(i) \neq w(i) \}$ and $w|_I \le v|_I$, then $w \le v$.
\end{lemma}

We note that one can obtain this lemma by repeatedly using \cite[Lemma 17]{billey2003maximal}. For the sake of completeness, we provide a proof at the end of this section.

The following corollary is an immediate consequence of Lemma \ref{lem:restriction}.
\begin{cor}\label{cor:bruhat_inv}
If $(i, j)$ is an inversion of $w$, then $w \cdot (i,j) < w$.
\end{cor}

\begin{remark}
We are tempted to conclude that $w\tau$ has exactly one less inversion than $w$ (i.e. that $w$ covers $w\tau$), but this is not true: if $w$ has one-line notation $321,$ then $(1, 3)$ an inversion of $w,$ but $w \cdot (1, 3)$ is the identity, which is not covered by $w.$ However, if $w$ is $321$-avoiding, we can easily show that $w$ must in fact cover $w\tau$.
\end{remark}

\subsection{The Temperley-Lieb Algebra and Non-Crossing Matchings}
In this subsection, we define the Temperley-Lieb algebra and discuss a few basic properties of the algebra, which we will then use to define the coefficients of Temperley-Lieb immanants. 


\begin{definition}(See \cite[\S 2]{LPP} and \cite[\S 3]{RS})
The \textbf{Temperley-Lieb algebra}, which we denote as $TL_n(2),$ is the $\mathbb{C}-$algebra generated by elements $t_1, t_2, \ldots, t_{n-1}$ with the relations $t_i^2 = 2t_i$ for $i \in [n],$ $t_it_j = t_jt_i$ for $i, j \in [n]$ where $|i - j| \geq 2,$ and $t_it_jt_i = t_i$ for $i, j \in [n]$ where $|i - j| = 1.$
\end{definition}
Notice that the Temperley-Lieb algebra in general depends on a parameter, but throughout this paper we specialize to the case when the parameter equals $2.$ The following result is standard. See \cite[Section 3]{RS}, for instance.

\begin{prop}\label{prop:tl}
The $\C$-linear map $\theta: \C(\mathfrak{S}_n) \to TL_n(2)$ defined by $\theta(s_i) = t_i-1$ is a $\C$-algebra homomorphism. The $\C$-linear map $\beta:$ $\C$-span of 321-avoiding permutations in $\fkS_n$ $\to TL_n(2)$ defined by $\beta(s_{i_1} s_{i_2} \cdots s_{i_k}) = t_{i_1} t_{i_2} \cdots t_{i_k}$ for any reduced word $s_{i_1} s_{i_2} \cdots s_{i_k}$ of a $321$-avoiding permutation in $\fkS_n$ is a well-defined vector space isomorphism.
\end{prop}

We use the elements of our Temperley-Lieb algebra in order to construct the first of our objects of interest, the Temperley-Lieb immanant.

\begin{definition}\label{syldef}
Given a permutation $u \in \mathfrak{S}_n,$ let $x_u$ be the monomial $\prod\limits_{i=1}^n x_{i, v(i)},$ living in the set of polynomials on $n^2$ variables $x_{i, j}$ for $1 \leq i, j \leq n.$ If we are furthermore given a 321-avoiding permutation $w \in \fkS_n,$ let $f_w (u)$ be the coefficient of $\beta(w)$ in $\theta(u)$. 
\par From here, for each 321-avoiding permutation $w \in \fkS_n,$ define the \textbf{Temperley-Lieb immanant} of $w$ by Imm$_{w} = \sum_{u \in \mathfrak{S}_n} f_w(u)x_u.$ This is the 
\end{definition}

A lot of the coefficients $f_w (u)$ can be easily computed (and, in particular, a lot are zero), which can be characterized by the following lemma.

\begin{lemma}[Proposition 3.7 in \cite{RS}]\label{lem:basic fwu}
Let $w, u \in \fkS_n$, with $w$ being $321$-avoiding.
\begin{enumerate}
    \item If $u \not\ge w$, then $f_w (u) = 0$.
    
    \item $f_w (w) = 1$.
\end{enumerate}
\end{lemma}

We can relate the coefficients $f_w(u)$ to colorings of a non-crossing matching, which we define below.

\begin{definition}\label{def:ncm}(See \cite[\S 2]{LPP})
A \textbf{non-crossing matching} of the integers in $[2n]$ is a perfect matching of these integers so that there doesn't exist $1 \le a < b < c < d \le 2n$ such that $a$ is paired with $c$ and $b$ is paired with $d$ in the matching.
\par For convenience, we sometimes denote the indices $n+1, n+2, \ldots, 2n$ in a non-crossing matching by $n', (n-1)', \ldots, 1',$ respectively. Also, let $[n]' = \{1', 2', \ldots, n'\}$ and $[a, b]' = \{a', (a+1)', \ldots, b'\}.$
\end{definition}

\begin{remark}
(1) The non-crossing aspect of the condition comes from the following interpretation: if we arrange the integers $1$ through $2n$ in that order around a circle, and draw a chord between paired vertices, then no two chords may intersect.

(2) If $v_1$ and $v_2$ are paired in a non-crossing matching, then there are an even number of vertices between $v_1$ and $v_2$. Thus, the difference of labels $v_2 - v_1$ must be odd. This parity consideration will be used extensively in this paper. \textbf{Warning:} in primed/unprimed notation, where $i' = 2n+1-i$, the difference of labels (when given labels in $[2n]$) between a primed vertex $j'$ and an unprimed vertex $i$ is $2n+1-j-i$, so $j-i$ must be even.
\end{remark}


There is a natural bijection between $321$-avoiding permutations of $[n]$ and non-crossing matchings of $2n$ vertices (see \cite[Section 3]{RS}). We summarize the bijection here:

\begin{prop}\label{prop:ncm}
There is a bijection, which we denote as $\ncm,$ between $321$-avoiding permutations of $[n]$ and non-crossing matchings of $2n$ vertices. Explicitly, for any $321$-avoiding permutation $w$, we construct a wiring diagram $D(w)$ corresponding to a reduced word for $w$, as in \cite{RS}. Fix a coordinate system such that each vertex $i$ is located at $(-1, n-i)$ and vertex $i'$ is located at $(1, n-i)$.

\begin{enumerate}[(a)]
\item For every $i$, the unique shortest path $p_i$ from $i$ to $w(i)'$ along the edges of $D(w)$ has either non-increasing, non-decreasing, or constant $y$-coordinate as a function of $x$, depending on whether $w(i) > i$, $w(i) < i$, or $w(i) = i$, respectively. Furthermore, two paths $p_i$ and $p_j$ intersect in at most one point. Let $S$ be the set of intersection points of some two $p_i, p_j$, $1 \le i < j \le n$.

    \item For each vertex $v$ (one of $i$ or $i'$ for some $i$) of the wiring diagram, there is a unique path $q_v$ along the edges of $D(w)$ starting at $v$ and terminating at another vertex (which we call $f(v)$), with the following properties:
\begin{itemize}
    \item the $y$-coordinate of $q_v$ is either non-increasing or non-decreasing as a function of path coordinate;
    
    \item for each $1 \le i \le n$, no positive length subsegment of the path $p_i$ intersects the path $q_v$ at exactly a single point. In other words, the path $q_v$ ``turns'' when it reaches a point of $S$.
\end{itemize}
Then the map $f$ induces a non-crossing matching of $2n$ vertices, which we call $\ncm(w)$.
\end{enumerate}
\end{prop}

\begin{remark}
We abuse notation slightly and let the pairs of a $321$-avoiding permutation $w$ be the pairs that occur in the corresponding non-crossing matching $\ncm(w)$, per the bijection established by Proposition \ref{prop:ncm}.
\end{remark}

For instance, the left figure below is the wiring diagram for the permutation $2341 = s_1 s_2 s_3$ following the convention of \cite{RS}. By following the procedure in Proposition \ref{prop:ncm}, we obtain the non-crossing matching corresponding to $2341$ in the right figure below.

\begin{center}
\includegraphics[]{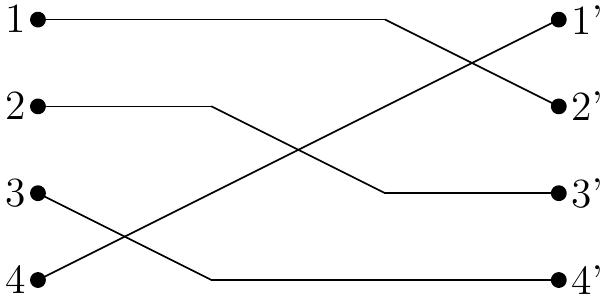}
\hspace{2cm}
\includegraphics[]{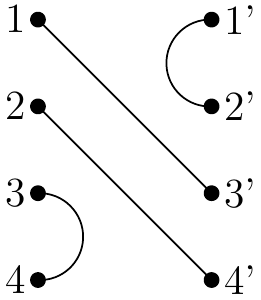}
\end{center}

\subsection{Complementary Minors and Colorings}
Temperley-Lieb immanants are closely related to complementary minors, which in turn can be represented by a coloring. We first describe colorings and how they are associated with complementary minors. 
\par For the purposes of this paper, a coloring of the vertices in $[n] \cup [n]' = \{1, 2, \ldots, n, 1', 2', \ldots, n'\}$ will always be a coloring of these vertices using the colors black and white.
\begin{definition}
A coloring of vertices $[n] \cup [n]'$ is \textbf{compatible} with a non-crossing matching if every black vertex is paired with a white vertex.
\par We also say that a coloring of vertices $[n] \cup [n]'$ is compatible with a permutation $w$ if the coloring is compatible with $\ncm(w).$
\end{definition}

\par In the below matching, for instance, the following is a compatible coloring.

\begin{center}
\includegraphics[]{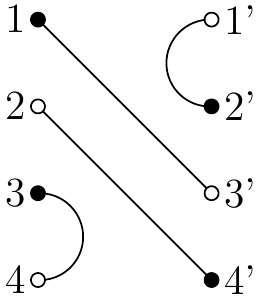}
\end{center}

\par We can represent any coloring by $(I, J)$, where $I \subset [n]$ is the set of unprimed black vertices, and $J' \subset [n]'$ is the set of primed white vertices, and thus say that a permutation $w$ is compatible with this pair $(I, J)$ of subsets if it is compatible with the corresponding coloring.
\par A coloring compatible with some non-crossing matching must have $n$ black and $n$ white vertices, which means $|I| = |J|$. As a result, we can associate the coloring $(I, J)$ to the complementary minor $\Delta_{I,J} \Delta_{\bar{I}, \bar{J}}$, defined below.

\begin{definition}
Consider the $n\times n$ matrix $(x_{i,j}),$ the matrix whose $(i, j)$ entry is the variable $x_{i, j}.$ For $I, J \subset [n]$ with $|I| = |J|$, we let $\Delta_{I,J}$ denote the determinant of the minor of this matrix with rows indexed by $I$ and columns indexed by $J$. Then the \textbf{complementary minor} of $I, J$ is $\Delta_{I,J} \Delta_{\bar{I}, \bar{J}}$.
\end{definition}
\par A key result we will use in this paper is a way to express a complementary minor as the sum of the Temperley-Lieb immanants compatible with the associated coloring.
\begin{prop}\cite[Proposition 4.3]{RS}\label{prop:cm equals imm sum}
Let $I,J \subseteq [n]$ with $|I| = |J|$, and consider the coloring $(I, J)$ (i.e. color $I, \bar{J}'$ black and $J', \bar{I}$ white). Then
\begin{equation}\label{cm equals imm sum}
    \Delta_{I,J}\Delta_{\bar{I}, \bar{J}} = \sum_{\substack{w \mathrm{\ compatible}\\ \mathrm{with\ } (I, J)}} \imm_w
\end{equation}
\end{prop}

\begin{remark}
Using the proof of Proposition 4.7 in \cite{RS}, if we treat the equations in \eqref{cm equals imm sum} for all $I, J \subset [n]$ as an (over-determined) linear system in the variables $\imm_w$, we see that the system has a unique solution. In fact, the idea of ``solving'' for $\imm_w$ as a linear combination of certain complementary minors will play a key role in this paper. However, while we know that complementary minors uniquely determine our Temperley-Lieb immanants, it is not entirely clear how to solve this system in general.
\end{remark}

\begin{example}
For $I = \{ 1 \}$ and $J = \{ 1 \}$, we have:
\begin{equation*}
        x_{1,1} \begin{vmatrix}
                x_{2,2} & x_{2,3} \\
                x_{3,2} & x_{3,3}
        \end{vmatrix} = \imm_{123} + \imm_{213}
    \end{equation*}
\begin{center}
    \includegraphics[]{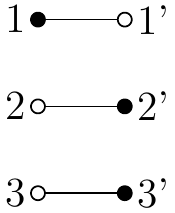}
    \hspace{2cm}
    \includegraphics[]{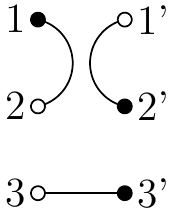}
\end{center}
Note that the left non-crossing matching corresponds to the permuation $123$ and the second corresponds to $213.$
\end{example}

\par In this paper, we also make the distinction between products of complementary minors as standalone determinants, and products of complementary minors that are ``embedded" in the matrix.

\begin{definition}\label{def:cm}
Define $\cm_{I,J}$ to be the determinant of the matrix $(y_{i,j})$ given by
\begin{equation*}
    y_{i,j} = \begin{cases}
        x_{i,j}, & i \in I, j \in J \text{ or } i \notin I, j \notin J, \\
        0, & \text{ otherwise.}
    \end{cases}
\end{equation*}
An explicit formula is
\begin{equation*}
    \cm_{I,J} = \sum_{\substack{u \in \fkS_n \\ u(I) = J}} \sgn(u) x_u.
\end{equation*}
\end{definition}

\begin{remark}
    Note that $x_u$ has nonzero coefficient in $\cm_{I,J}$ iff for all $1 \le i \le n$, the vertices $i$ and $u(i)'$ are assigned different colors in the coloring $(I, J)$.
\end{remark}

\begin{lemma}\label{cmminors}
We have $\cm_{I,J} = (-1)^{s(I) + s(J)}\Delta_{I,J} \Delta_{\overline{I},\overline{J}}$, where $s(I) := \sum_{i \in I} i$.
\end{lemma}
\par As an example, consider the product of complementary minors $$\Delta_{1,4}\Delta_{234,123} = \begin{vmatrix} x_{1,4} \end{vmatrix} \begin{vmatrix} x_{2,1} & x_{2,2} & x_{2,3} \\ x_{3,1} & x_{3,2} & x_{3,3} \\ x_{4,1} & x_{4,2} & x_{4,3} \end{vmatrix},$$ where $I = \{1\}, J = \{4\}.$ Observe that $x_{1,4}x_{2,3}x_{3,2}x_{4,1}$ has coefficient $-1$ in the expansion of $\Delta_{1,4}\Delta_{234,123}$. However, if we instead represented the product of complementary minors by zeroing out entries of the $4 \times 4$ determinant, we would have the following: $$\begin{vmatrix} 0 & 0 & 0 & x_{1,4} \\ x_{2,1} & x_{2,2} & x_{2,3} & 0 \\ x_{3,1} & x_{3,2} & x_{3,3} & 0 \\ x_{4,1} & x_{4,2} & x_{4,3} & 0\end{vmatrix}.$$ Notice here that the coefficient of $x_{1,4}x_{2,3}x_{3,2}x_{4,1}$ is $+1$, and agrees with what we'd expect if we instead took the normal determinant.
\par As such, when we represent the complementary minors by pictures, we refer to this latter picture in which the minors are viewed as being within the larger matrix, rather than standing alone as matrices. The signs in front of the pictures will then correspond to adding or subtracting the corresponding $\cm_{I, J}.$
\par The pictoral representation we use for this complementary minor is the following. 
\begin{center}
\includegraphics[]{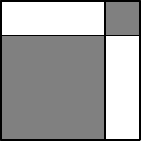}
\end{center}
In this case we denote by the shaded region the entries that we include in our complementary minor, and the white areas as the regions that are zeroed out. This convention is followed throughout the rest of the paper.

\par Finally, we record two miscellaneous lemmas that will be useful for us. We first note that the matching diagram admits certain symmetries. Specifically, reflecting a non-crossing matching $\ncm(w)$ about either axis of the diagram also gives a non-crossing matching corresponding to either $\ncm(w^{-1})$ or $\ncm(w_0 ww_0)$, where $w_0$ is the longest word in $\fkS_n$.
\begin{lemma}\label{lem:symmetry}
If $w, u \in \mathfrak{S}_n$ with $w$ $321$-avoiding, then
$f_w (u) = f_{w^{-1}} (u^{-1}) = f_{w_0 w w_0} (w_0 u w_0)$.
\end{lemma}

\par Given a non-crossing matching, we can also construct a compatible coloring with some nice properties.

\begin{lemma}\label{lem:matching_compatibility}
Let $w$ be a $321$-avoiding permutation. Define a coloring on $[n]\cup[n]'$ as follows: for each $1 \le i \le n$, color $i$ black and $w(i)'$ white if $w(i) > i$; color $i$ white and $w(i)'$ black if $w(i) < i$; and arbitrarily color $i$ and $w(i)'$ one white, one black if $w(i) = i$. Then, in $\ncm(w)$, every white vertex $i$ or $i'$ is paired with a black vertex $j$ or $j'$ with $j \leq i$. In particular, $w$ is compatible with the coloring.
\end{lemma}
\begin{remark}
Proposition \ref{prop:ncm} gives a way to explicitly construct $\ncm(w)$ given the one-line notation of $w$. In practice, it is sufficient to know properties of $\ncm(w)$ given in Lemma \ref{lem:matching_compatibility}.
\end{remark}

\subsection{Proofs of Lemmas in Section 2}
In this section, we restate all lemmas in Section 2 and provide their proofs. 

\begin{replemma}{lem: restriction inverses}
Let $w \in \fkS_n.$ Then, if $v = w|_I \in \fkS_m$ for some subset $I \subset [1, n]$ (which in particular requires $|I| = m$), then $v^{-1} = (w^{-1})|_{w(I)}.$ In particular, $w$ avoids the pattern $v$ if and only if $w^{-1}$ avoids the pattern $v^{-1}$.
\end{replemma}

\begin{proof}
Fix $i \in [1, m].$ Let $x$ be the $i$-th smallest element in $w(I)$, and suppose $w^{-1}(x)$ is the $j$-th smallest element of $I$. Then since $w^{-1}(w(I)) = I$ (as $w$ is a permutation, and thus bijection), we have that $(w^{-1})|_{w(I)}(i) = j.$ However, we know that $w$ sends $w^{-1}(x)$ to $x,$ meaning that $w|_I(j) = i,$ or that $v(j) = i.$ Hence, for this $i,$ we have that $v^{-1}(i) = (w^{-1})|_{w(I)}(i).$ But this holds for all $i,$ giving us the desired result.
\end{proof}

\begin{replemma}{lem: restriction flips}
Let $w \in \fkS_n.$ Then, if $v = w|_I \in \fkS_m$ for some subset $I \subset [1, n],$ and $w_0 \in \fkS_n, w_0' \in \fkS_m$ are the longest words in their respective permutation groups, then $w_0'v = (w_0w)|_I$ and $vw_0' = (ww_0)|_{w_0(I)}$.
\end{replemma}

\begin{proof}
First, fix an arbitrary $i \in [1, m].$ Then, if $x$ is the $i$-th smallest element of $I$, suppose it gets sent to the $j$-th smallest element of $w(I).$ By definition of $w|_I$, we have $v(i) = j$. Since $w_0$ has one line notation $n(n-1)(n-2)\ldots 1$, we have $x < y \iff w_0(x) > w_0(y),$ for $x, y \in [1, n].$ Therefore, $w_0w$ sends the $i$-th smallest element of $I$ to the $j$-th largest element of $w_0w(I),$ so $(w_0w)|_I(i) = m + 1 - j.$ But notice that $w_0'(x) = m + 1 - x$ for each $x \in [1, m],$ and $j = v(i).$ Thus, we have that $w_0'v(i) = (w_0w)|_I(i).$ But $i$ is arbitrary, giving the first part of the lemma.
\par The second part proceeds similarly: for each $i \in [1, m],$ if $x$ is the $i$-th smallest element of $w_0(I),$ then $w_0x$ is the $i$-th largest element of $I,$ or the $(m + 1 - i)$-th smallest element of $I.$ Then, say that $j$ is so that $w(w_0(x))$ is the $j$-th smallest element of $w(I).$ We thus have that $(ww_0)|_{w_0(I)}(i) = j.$ Meanwhile, notice that $w$ sends the $(m + 1 - i)$-th smallest element of $I$ to the $j$-th smallest element of $w(I)$, meaning that $vw_0'(i) = v(m + 1 - i) = j.$ But this holds for each $i,$ so $vw_0' = (ww_0)|_{w_0(I)},$ finishing the proof of the lemma.
\end{proof}

\begin{replemma}{lem:restriction}
Let $v, w \in \fkS_n$. If $I \supset \{ i \mid v(i) \neq w(i) \}$ and $w|_I \le v|_I$, then $w \le v$.
\end{replemma}

\begin{proof}
Using equivalent definition \ref{bruhat3} of the Bruhat order, our goal is to show that $|w([1,i]) \cap [1,j]| \ge |v([1,i]) \cap [1,j]|$ for positive integers $i, j$ with $1 \le i, j \leq n.$ To show this, we add up the following two equations:
\begin{itemize}
    \item $w([1,i] \backslash I) \cap [1,j] = v([1,i] \backslash I) \cap [1,j]$ because $w(k) = v(k)$ for $k \notin I$;

    \item $|w([1,i] \cap I) \cap [1,j]| \ge |v([1,i] \cap I) \cap [1,j]|$ follows by applying the equivalent definition \ref{bruhat3} of Bruhat order to $w|_I \le v|_I$. \qedhere
\end{itemize}
\end{proof}

\begin{replemma}{lem:basic fwu}
Let $w, u \in \fkS_n$, with $w$ being $321$-avoiding.
\begin{enumerate}
    \item If $u \not\ge w$, then $f_w (u) = 0$.
    
    \item $f_w (w) = 1$.
\end{enumerate}
\end{replemma}

\begin{proof}
Let $w = s_{i_1} s_{i_2} \cdots s_{i_k}$ and $u = s_{j_1} s_{j_2} \cdots s_{j_m}$ be reduced words. Consider expanding $\theta(u) = (t_{j_1} - 1)(t_{j_2} - 1) \cdots (t_{j_m} - 1)$ and simplifying each monomial using the Temperley-Lieb algebra relations. The resulting monomial is a subword of the original monomial, which turn is a subword of $t_{j_1} t_{j_2} \cdots t_{j_m}$. Thus, we can only have a $t_{i_1} t_{i_2} \cdots t_{i_k}$ term in the expansion if $s_{i_1} s_{i_2} \cdots s_{i_k}$ is a subword of $s_{j_1} s_{j_2} \cdots s_{j_m}$. Thus, if $f_w (u) \neq 0$, then by equivalent definition \ref{bruhat1} of Bruhat order, we get $u \ge w$, proving part 1 of the Lemma. For part 2, we take $w = u$, and observe that when expanding $(t_{i_1} - 1)(t_{i_2} - 1) \cdots (t_{i_k} - 1)$, the only way to get a monomial of length $k$ is to take the first term of each binomial, and thus the coefficient of $t_{i_1} t_{i_2} \cdots t_{i_k}$ is one.
\end{proof}

\begin{replemma}{cmminors}
We have $\cm_{I,J} = (-1)^{s(I) + s(J)}\Delta_{I,J} \Delta_{\overline{I},\overline{J}}.$
\end{replemma}

\begin{proof}
This is clearly true if $I = J = [k]$ for some $k$. In the general case, notice that we can reduce to this case by performing $\sum_{i \in I} (i-1) = s(I) - |I|$ many row swaps and $\sum_{j \in J} (j-1) = s(J) - |J|$ many column swaps on the matrix $M$ of $\cm_{I,J}$. To conclude, we note that $|I| = |J|,$ and each row or column swap changes the sign of $M$.


    
    

\end{proof}

\begin{replemma}{lem:symmetry}
If $w, u \in \mathfrak{S}_n$ with $w$ $321$-avoiding, then
$f_w (u) = f_{w^{-1}} (u^{-1}) = f_{w_0 w w_0} (w_0 u w_0)$.
\end{replemma}
\begin{proof}
\par Throughout the proof of this lemma, given $I, J \subset [n],$ we say that $w$ is compatible with $(I, J)$ if $\ncm(w)$ compatible with the coloring where $I, \overline{J}'$ is colored black, and $\overline{I}, J'$ is colored white.
\par From Proposition \ref{prop:cm equals imm sum}, we have for all $I, J \subset [n]$ with $|I| = |J|$ and $u \in \fkS_n$,
\begin{equation}\label{cmeqimm}
    \sgn(u) \one_{u(I) = J} = \sum_{w \mathrm{\ compatible\ with\ } (I, J)} f_w (u),
\end{equation}
where $\one_{u(I) = J}$ is the indicator function that is $1$ on $u$ if $u(I) = J$ and $0$ otherwise.
\par We leverage the fact (see Remark after Proposition \ref{prop:cm equals imm sum}) that for a given $u$, this linear system of equations has a unique solution in the $\{ f_w (u) \}$.
Note that $w$ is compatible with $(I, J)$ iff $w^{-1}$ is compatible with $(J, I)$, and $\sgn(u^{-1}) = \sgn(u)$, so we can rewrite \eqref{cmeqimm} as
\begin{equation*}
    \sgn(u^{-1}) \one_{u^{-1} (J) = I} = \sum_{w \mathrm{\ compatible\ with\ } (J, I)} f_{w^{-1}} (u).
\end{equation*}
Replacing $u$ with $u^{-1}$ and swapping $I, J$, we get (for all $I, J \subset [n]$ with $|I| = |J|$ and $u \in \fkS_n$):
\begin{equation}\label{cmeqimm2}
    \sgn(u) \one_{u(I) = J} = \sum_{w \mathrm{\ compatible\ with\ } (I, J)} f_{w^{-1}} (u^{-1}).
\end{equation}
Notice that \eqref{cmeqimm} and \eqref{cmeqimm2} are the same system of equations with different variables $\{ f_w (u) \}$ and $\{ f_{w^{-1}} (u^{-1}) \}$. By uniqueness of solution to \eqref{cmeqimm}, we have $f_w (u) = f_{w^{-1}} (u^{-1})$.

A similar argument holds to show $f_w (u) = f_{w_0 w w_0} (w_0 u w_0)$. Define $I^* = \{ n+1 - i \mid i \in I \}$. Note that $w$ is compatible with $(I, J)$ iff $w_0 w w_0$ is compatible with $(I^*, J^*)$, so
\begin{equation*}
    \sgn(w_0 u w_0) \one_{w_0 u w_0 (I^*) = J^*} = \sum_{w \mathrm{\ compatible\ with\ } (I^*, J^*)} f_{w_0 w w_0} (u).
\end{equation*}
But this holds for all $I^*, J^*$, and $u$, so by uniqueness of solution to \eqref{cmeqimm}, we have $f_w (u) = f_{w_0 w w_0} (w_0 u w_0)$.

An interested reader is encouraged to find an alternate proof of Lemma \ref{lem:symmetry} directly using Definition \ref{syldef}.
\end{proof}

\begin{replemma}{lem:matching_compatibility}
Let $w$ be a $321$-avoiding permutation. Define a coloring on $[n] \cup [n]'$ as follows: for each $1 \le i \le n$, color $i$ black and $w(i)'$ white if $w(i) > i$; color $i$ white and $w(i)'$ black if $w(i) < i$; and arbitrarily color $i$ and $w(i)'$ one white, one black if $w(i) = i$. Then, in $\ncm(w)$, every white vertex $i$ or $i'$ is paired with a black vertex $j$ or $j'$ with $j \leq i$. In particular, $w$ is compatible with the coloring.
\end{replemma}

\begin{proof}
\par We use the notation of Proposition \ref{prop:ncm}. Additionally, let $y(v)$ be the $y$-coordinate of a vertex. Note that $y(i) = y(i') = n-i$. Thus, we want to show two claims: if $v$ is white, then (1) $f(v)$ is black and (2) $y(v) \le y(f(v))$.

Say a vertex $i$ or $i'$ is fixed if $w(i) = i$. If $v = i$ or $i'$ is fixed, then by Proposition \ref{prop:ncm}(a), $i$ is paired with $i'$, so $f(v) = i'$ or $i$ respectively. Thus, since $v$ is white, our coloring says that $f(v)$ must be black. Thus, we may assume $v$ is not fixed; then neither is $f(v)$.

Given a non-fixed vertex $v$, let $p_v$ be the path $p_i$ if $v = i$ is unprimed and the path $p_{w^{-1} (i)}$ if $v = i'$ is primed. We reformulate the coloring as follows: a non-fixed vertex $v$ is white if and only if when we walk along the path $p_v$ starting from $v$, the path is non-decreasing in $y$-coordinate.

Suppose we start from a non-fixed white vertex $v$, and walk along the piecewise linear path $q_v$ representing the pairing involving this white vertex. Let $p_v$ be the path $p_i$ if $v = i$ is unprimed and the path $p_{w^{-1} (i)}$ if $v = i'$ is primed. Since $v$ is a white vertex, the path $p_v$ is non-decreasing in $y$-coordinate if we walk along it starting at $v$. Since a walk along $q_v$ from $v$ starts by walking along $p_v$, we see by Proposition \ref{prop:ncm}(b) that the path $q_v$ is always non-decreasing in $y$-coordinate when walking from $v$ to $f(v)$. In particular, we have $y(v) \le y(f(v))$, proving our claim 2. Now once we get to $f(v)$, we turn around, and then our path will be non-increasing in $y$-value. But then we will be walking along $q_{f(v)}$, and thus by our coloring rule, we see that $f(v)$ must be black, proving our claim 1. This finishes the proof of the lemma.
\end{proof}

\section{\%-immanants}

\par We now define another subclass of immanant that we study in this paper, which we call \%-immanants. These immanants capture the idea of permutations fitting within skew-tableau lying in $n \times n$ matrices, which has been studied previously in connection to the Bruhat order and Kazhdan-Lusztig immanants. 
\par For instance, \cite{MR2353118} discusses the order ideals of a permutation and how these are related to a `Skew Ferrers Matrix.' In particular, \cite{MR2353118} describes when the set of permutations that fit within a certain skew-tableau in an $n \times n$ matrix is a principal ideal in the Bruhat order. Additionally \cite{CSB}, Chepuri and Sherman-Bennett discuss the ``determinantal formulas" arising from taking the determinant of a matrix where, outside of a skew-tableau, all the entries in the matrix are zero. We can view these formulas as immanants (in fact, our \%-immanants), as such a determinant is a $\mathbb{C}-$linear combination of the monomials $x_{\sigma},$ running over permutations $\sigma \subset \fkS_n.$ 
\par In this section, we define \%-immanants and provide a couple of simple examples. The main result of this section is a classification of the space of immanants generated by these \%-immanants. We re-define a \%-immanant here for convenience.
\begin{definition}
For a skew tableau $\lambda / \mu$, where $\lambda = (\lambda_1, \lambda_2, \ldots, \lambda_n)$ and $\mu = (\mu_1, \mu_2, \ldots, \mu_n)$ are non-increasing sequences of non-negative integers, we say a permutation $\sigma \in \fkS_n$ \textbf{lies in} $\lambda/\mu$ if for all $1 \le i \le n$, we have $\mu_i < \sigma(i) \leq \lambda_i.$ (Geometrically, we have $(i, \sigma(i)) \in \lambda/\mu$.)

Let $A$ is the set of permutations in $\fkS_n$ that lie in $\lambda/\mu$, then define $\imm^{\%}_{\lambda/\mu} = \sum_{\sigma \in A} \sgn(\sigma) x_{\sigma}$. We refer to this polynomial $\imm^{\%}_{\lambda/\mu}$ as a \textbf{\%-immanant} of degree $n$. We say the \textbf{diagram} of $\imm^{\%}_{\lambda/\mu}$ is  $\lambda/\mu$, embedded into a $n \times n$ square.
\end{definition}

Throughout this paper, we orient the $n \times n$ bounding box such that $(1, 1)$ is the upper-left unit square and $(n, 1)$ is the lower-left unit square. For instance, the following is the diagram of $\imm^\%_{(5, 5, 3, 2, 2)/(2, 1)}$ (for convenience, we drop trailing zeroes from $\lambda, \mu$). The areas within the skew shape are shaded. Note that this corresponds to diagrams of left-aligned skew Ferrers matrices in \cite{MR2353118}.
\begin{center}
\includegraphics[]{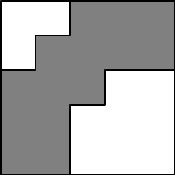}
\end{center}

A certain subset of \%-immanants are naturally associated to a permutation $w$ as follows.
\begin{definition}\label{defn:per imm w}
Suppose $w \in \mathfrak{S}_n.$ Let $m_w(i) = \min w([1, i])$ and let $M_w(i) = \max w([i, n]).$ Define $\mu, \lambda$ by $\mu_i = m_w(i)-1$ and $\lambda_i = M_w(i)$ for $1 \le i \le n$. Then, we define $\imm_w^{\%} = \imm_{\lambda/\mu}^{\%}.$
\end{definition}
For instance, the following is the diagram of $\imm^\%_{2143}$, where the positions of $(i, w(i))$ are marked with an ``X," and the areas within the skew shape are shaded. This is also referred to as the left hull of the permutation by \cite{MR2353118}, for instance.
\begin{center}
\includegraphics[]{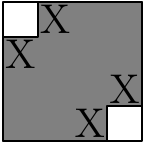}
\end{center}

Now, let $P_n^{\%}$ be the vector subspace of the space of immanants spanned by the \%-immanants of degree $n$. In order to classify this space, we need the following definition.

\begin{definition}\label{defn:1234-1324}
Two permutations $w, w' \in \fkS_n$ are \textbf{$1324$-adjacent} if $w(i) = w'(i)$ for all but two values $a < b,$ and there exist values $c, d$ with $c < a < b < d,$ such that either $w(c) < w(a) < w(b) < w(d)$ and $w'(c) < w'(b) < w'(a) < w'(d),$ or vice versa for $w, w'.$
\par Two permutations $w, w' \in \fkS_n$ are then said to be \textbf{$1324$-related} if there exists a sequence of permutations $w = w_0, w_1, \ldots, w_m = w' \in \fkS_n$ so that each adjacent pair of permutations are $1324$-adjacent. Being $1324$-related forms an equivalence relation on the permutations in $\fkS_n.$
\end{definition}
For instance, $14235$ is $1324$-adjacent to $13245,$ which in turn is $1324$-adjacent to $12345.$ In fact, one can show that all permutations $w \in \fkS_5$ where $w(1) = 1, w(5) = 5$ are $1324$-related.  
\par The idea behind the name of $1324$-related is that permutations that are $1324$-related can be obtained from each other by turning $1324$ patterns into $1234$ patterns (and vice versa).
\par This condition turns out to be heavily related with coefficients of immanants lying in $P_n^{\%}.$
\begin{thm}\label{thm:classifying space of percent}
The following conditions are equivalent:
\begin{enumerate}
    \item The immanant $\imm_f$ lies in $P_n^{\%},$
    \item The function $f$ satisfies $f(w) = -f(w'),$ for all $w, w' \in \fkS_n$ that are $1324$-adjacent.
\end{enumerate}
\end{thm}
In order to prove this, we first prove three intermediate results. Assume $n$ is fixed throughout.

\begin{lemma}\label{lem:bigtableau}
If $\imm^\%_{\lambda/\mu} \neq 0$, then $(i, n+1-i) \in \lambda/\mu$ for all $1 \le i \le n$.
\end{lemma}

\begin{proof}
We show the contrapositive assertion. If $(i, n+1-i) \notin \lambda/\mu$ for some $i$, then either $(a, b) \notin \lambda/\mu$ for all $a \le i, b \le n+1-i$ or $(a, b) \notin \lambda/\mu$ for all $a \ge i, b \ge n+1-i$. In the first case, for any permutation $w \in \fkS_n$, there must exist $j \le i$ such that $w(j) \le n+1-i.$ Thus, $(j, w(j)) \notin \lambda/\mu$, and so $\imm^\%_{\lambda/\mu} = 0$. The second case is analogous.
\end{proof}

We can define a partial order on immanants: we say $\imm_{\lambda_1/\mu_1}^\% \ge \imm_{\lambda_2/\mu_2}^\%$ if $\lambda_1/\mu_1$ contains $\lambda_2/\mu_2$.

\begin{lemma}\label{lem:engulfing}
    We have $w \in \fkS_n$ lies in $\lambda/\mu$ iff $\imm_w^\% \le \imm_{\lambda/\mu}^\%$.
\end{lemma}

\begin{proof}
    Let $\imm_w^\% = \imm_{\lambda_w/\mu_w}^\%$. If $w \notin \lambda/\mu$, then there exists $1 \le t \le n$ such that $(t, w(t)) \notin \lambda/\mu$. But $(t, w(t)) \in \lambda_w/\mu_w$, so $\lambda_w/\mu_w \not\subset \lambda/\mu$ and $\imm_w^\% \not\le \imm_{\lambda/\mu}^\%$.

    Now suppose $w \in \lambda/\mu$; then $\mu_i < w(i) \le \lambda_i$ for all $1 \le i \le n$. We claim that $\mu_i < m_w (i)$ for all $1 \le i \le n$. Indeed, if $m_w (i) = w(j)$ for some $1 \le j \le i$, then $m_w (i) = w(j) > \mu_j \ge \mu_i$. Similarly, we have $M_w (i) \le \lambda_i$ for all $1 \le i \le n$. Thus, $\lambda_w/\mu_w \subset \lambda/\mu$ and $\imm_w^\% \le \imm_{\lambda/\mu}^\%$.
\end{proof}

\begin{prop}\label{prop:same percent}
The permutations $w, w' \in \fkS_n$ are $1324$-related if and only if $\imm_w^{\%} = \imm_{w'}^{\%}.$
\end{prop}

\begin{proof}
First, suppose that the permutations are $1324$-related, so there exists a sequence of permutations where each adjacent pair of them are $1324$-adjacent. It suffices to show $\imm_w^{\%} = \imm_{w'}^{\%}$ when $m = 1,$ because the general case follows from the $m=1$ case and transitivity. 
Turning to the $m=1$ case, and swapping $w$ and $w'$ if necessary, we conclude there must exist $1 \leq a < b < c < d \leq n$ so that $w(a) < w(b) < w(c) < w(d)$ and $w'(a) < w'(c) < w'(b) < w'(d).$ 
\par To show that $\imm_w^{\%} = \imm_{w'}^{\%},$ we need to prove that $m_w(i) = m_{w'}(i)$ and $M_w(i) = M_{w'}(i)$ for all $i.$ First, notice that $w(j) = w'(j)$ for all $j$ except for $j = b$ and $j = c,$ by definition \ref{defn:1234-1324}. Therefore, if $i < b$ or $i \geq c,$ then $m_w(i) = m_{w'}(i)$ and $M_w(i) = M_{w'}(i)$ because these are the minimum and maximum, respectively, of the same set.
\par Otherwise, if $b \leq i < c,$ since $w(a) = w'(a) < w(b) < w'(b),$ it follows that $m_w(i) \leq w(a) < w(b), w'(b).$ But if $m_w(i) = w(j_i),$ then by definition \ref{defn:1234-1324} this equals $w'(j_i),$ as $j_i \neq b.$ Therefore, we have $m_w(i) \geq m_{w'}(i).$ Swapping the roles of $w, w'$ and repeating the argument yields $m_w(i) \leq m_{w'}(i),$ and thus these two values are equal.
\par Similarly, $M_w(i), M_{w'}(i) \geq w(d) > w(c), w'(c),$ so if $M_w(i) = w(j_i),$ as $j_i$ cannot equal $b$ or $c,$ we have $$M_w(i) = w(j_i) = w'(j_i) \geq M_{w'}(i).$$ Again, we can swap the roles of $w, w'$ to get our equality for all $i.$ This implies that $\imm_w^{\%} = \imm_{w'}^{\%},$ which we wanted.
\par For the other direction, suppose that $\imm_w^{\%} = \imm_{w'}^{\%}.$ We will show the existence of a sequence starting with $w$ and ending with $w'$ and where each pair of adjacent permutations in the sequence are $1324$-adjacent. We induct on the largest value $i$ such that $w(i) \neq w'(i),$ if one exists (and set $i = 0$ if this never holds). Our base case is $i = 0,$ where $w = w'$ and this is obvious.
\par Suppose we have shown that this holds for $i = 0, 1, 2, \ldots, j - 1,$ and $w, w'$ are permutations such that $\imm^{\%}_{w} = \imm^{\%}_{w'},$ and $j$ is the largest value such that $w(j) \neq w'(j).$ In particular, we have that $w(i) = w'(i)$ for $i > j.$ Without loss of generality, suppose that $w(j) > w'(j).$ Let $w'' = (w'(j), w(j)) \cdot w,$ and let $k$ be such that $w(k) = w'(j);$ note that $k < j.$ We will show that $w, w''$ are $1324$-adjacent.
\par To do this, first observe that $w(x) = w''(x)$ for $x \neq j, k.$ The main difficulty in this proof is finding values $a, d$ such that $a < k < j < d$ and $w(a) < w(j), w(k) < w(d).$ 
\par To do this, by assumption, since $\imm_{w}^{\%} = \imm_{w'}^{\%},$ we require $$M_{w'}(j) = M_{w}(j)$$ and $$m_{w'}(k) = m_{w}(k).$$ We first claim that $w^{-1}(m_w(k)) < k$ and $j < w^{-1}(M_w(j)).$
\par First, note that since $M_w (j) = \max w([j, n])$, we have $w^{-1}(M_w(j)) \geq j$. Suppose that this is an equality; then $M_{w'} (j) = M_w (j) = w(j).$ Since $M_{w'}(j) = \max w'([j, n])$, we have $w(j) \in w'([j, n])$. But observe that $w(j) \neq w(i) = w'(i)$ for $i > j,$ and $w(j) \neq w'(j),$ so $w(j)$ cannot lie in $w'([j, n])$, a contradiction. Thus, $w^{-1}(M_w(j)) > j.$
\par Similarly, $w^{-1}(m_w(k)) \leq k.$ If this is an equality, we have $\min w'([1, k]) = m_{w'} (k) = m_w (k) = w(k) = w'(j).$ But this is impossible, since $j \not \in [1, k]$. Thus, $w^{-1}(m_w(k)) < k.$
\par We now argue that $m_w(k) < w(k) < w(j) < M_w(j).$ By assumption, we have $w(j) > w(k).$ Furthermore, $w(k) \in w([1, k]),$ so $k \neq w^{-1}(m_w(k))$ implies that $w(k) > m_w(k).$ Similarly, $j \neq w^{-1}(M_w(j))$ and $w(j) \in w([j, n]),$ so $w(j) < M_w(j).$
\par Hence, we have that $w^{-1}(m_w(k)) < k < j < w^{-1}(M_w(j)),$ and $m_w(k) < w(k) = w'(j) < w(j) < M_w(j).$ In particular, it follows that $w$ and $w'' = (w'(j), w(k))w$ are $1324$-adjacent, so by the other direction of this proposition, it follows that $\imm_{w''}^{\%} = \imm_w^{\%} = \imm_{w'}^{\%}.$ Furthermore, by construction, we have that $w''(i) = w(i) = w'(i)$ for $i > j,$ and $w''(j) = w(k) = w'(j),$ meaning that the largest value where $w''$ and $w'$ disagree is less than $j.$ By the inductive hypothesis, there exists a sequence of permutations $w'', w_1, w_2, \ldots, w_k, w'$ where each adjacent pair is $1324$-adjacent. But then $w, w'', w_1, w_2, \ldots, w_k, w'$ is a sequence of permutations where each pair is $1324$-adjacent, and therefore $w, w'$ are $1324$-related. This proves the proposition.
\end{proof}
\begin{remark}
This proposition is the first hint that $\%$-immanants aren't able to distinguish between different $1324$-related permutations (in the sense of statement 2 of Theorem \ref{thm:classifying space of percent}). This idea is also an intuitive explanation for why, when asking which Temperley-Lieb immanants are linear combinations of $\%$-immanants later in the paper, avoiding the pattern $1324$ would be necessary. We will make these ideas more precise in later sections of the paper, once we understand some properties of the coefficients of Temperley-Lieb immanants.
\end{remark}

We now are ready to prove the theorem.

\begin{proof}[Proof of Theorem \ref{thm:classifying space of percent}]
Let $V$ be the vector space of immanants $\imm_f$ for which $f(w) = -f(w')$ for all $w, w'$ that are $1324$-adjacent. We want to show $P_n^{\%} = V$.

First, we prove $P_n^{\%} \subset V$. Since $P_n^{\%}$ is spanned by $\%$-immanants, it suffices to check that $\imm_{\lambda/\mu}^\% \in V$ for all skew tableaux $\lambda/\mu$. Fix $\lambda/\mu$, and for any $w \in \fkS_n$, let $c_w$ denote the coefficient of $x_w$ in $\imm^{\%}_{\lambda/\mu}$. In order to show $\imm_{\lambda/\mu}^\% \in V$, it suffices to check $c_w = -c_{w'}$ for any $w, w' \in \fkS_n$ that are $1324$-adjacent. Note that $c_w = \sgn(w)$ if $w$ lies in $\lambda/\mu$ and $0$ otherwise, and likewise for $w'$. Since $\sgn(w) = -\sgn(w')$, we obtain that $c_w = -c_{w'}$ if we can show the following:

\textbf{Claim.} $w$ lies in $\lambda/\mu$ iff $w'$ lies in $\lambda/\mu$.

By Lemma \ref{lem:engulfing} and Proposition \ref{prop:same percent}, both statements are equivalent to $\imm_w^\% \le \imm_{\lambda/\mu}^\%$. (A direct proof is also possible.)

\par This proves the Claim and thus we have shown that $P_n^{\%} \subset V$. Now, we will prove that $\dim V = \dim P_n^{\%}$.
\par Let $W$ be a set of representatives for the $1324$-related equivalence classes. For each $w \in W$, let $I_w$ be the set of permutations $w'$ that are $1324$-related to $w,$ and define the immanant $\chi_{I_w} = \sum_{\sigma \in I_w} \sgn(\sigma) x_\sigma$. First, note that the $\chi_{I_w},$ ranging over $w \in W$, span $V$. Indeed, if some immanant $\sum\limits_{w \in \fkS_n} f(w) x_w$ lies in $V,$ then $\sum\limits_{w \in \fkS_n} f(w) x_w =  \sum\limits_{w \in W} f(w) \sgn(w) \chi_{I_w},$ since any other permutation $w' \in \fkS_n$ is $1324$-related to some element $w \in W,$ meaning that $f(w') = \sgn(w)\sgn(w') f(w).$ Thus, $\dim V \le |W|$.

\par Next, we show that $\{ \imm_w^\% \}_{w \in W}$ are linearly independent in $P_n^\%$. Indeed, suppose that there is a subset $W' \subset W$ such that $\sum_{w \in W'} a_w \imm_w^\% = 0$ for nonzero $a_w$. By Proposition \ref{prop:same percent}, we have $\imm_w^\% \neq \imm_v^\%$ for distinct $v, w \in W$. Thus, there is some maximal element $v \in W'$, in the sense that $\imm_w^\% \not\le \imm_v^\%$ for $w \in W'$, $w \neq v$ (using the partial order defined before Lemma \ref{lem:engulfing}). Now consider the coefficient of $x_v$ in $\sum_{w \in W'} a_w \imm_w^\%$. The coefficient of $x_v$ in $\imm_v^\%$ is $\sgn(v)$, but the coefficient of $x_v$ in $\imm_w^\%$ is $0$ by Lemma \ref{lem:engulfing} and maximality of $v$. Thus, we must have $a_v \sgn(v) = 0$, a contradiction since by assumption $a_v \neq 0$. Hence, $\{ \imm_w^\% \}_{w \in W}$ are linearly independent in $P_n^\%$, so $\dim P_n \ge |W|$.

\par Thus, $|W| \le \dim P_n^\% \le \dim V \le |W|$. As a result, $\dim P_n^\% = \dim V$. Since $P_n^\% \subset V$, we have $P_n^\% = V$ and this completes the proof of the theorem. \qedhere

\end{proof}

\section{Temperley-Lieb Immanants as One \%-Immanant}\label{sec:specific-TL}
In this section, we classify the $321$-avoiding permutations $w$ whose Temperley-Lieb immanant is a \%-immanant up to sign. 
\begin{thm}\label{thm:onepercentimm}
Let $w$ be a $321$-avoiding permutation. Then $\imm_w$ is a \%-immanant up to sign if and only if $w$ avoids both $1324$ and $2143$. In that case, $\imm_w = \sgn(w) \imm^\%_w$.
\end{thm}

\subsection{2143-, 1324-Avoiding Implies $\imm_w = \sgn(w)\imm_w^{\%}$}
This subsection is devoted to proving the if direction of Theorem \ref{thm:onepercentimm}. This is a special case of \cite[Corollary 3.6]{CSB} which drops the assumption that $w$ is $321$-avoiding, whose proof is based on results on Schubert varieties. Here, we present a purely combinatorial proof, using the relationship between Temperley-Lieb immanants, non-crossing matchings, and complementary minors described in \cite{RS} and summarized in Section \ref{sec:prelims}. The ideas in this section will also be used extensively in Section \ref{sec:general_TL}. 

\begin{thm}\label{CSBThm}Special case of \cite[Corollary 3.6]{CSB}
If $w$ is a permutation that avoids the patterns 321, 1324, and 2143, then $\imm_w$ is a \%-immanant up to sign. Specifically, $\imm_w = \sgn(w) \imm_w^{\%}.$
\end{thm}

The main method of the proof is as follows. First, we show that $\imm_w^\%$ can be nicely expressed as a sum of complementary minors, each of which is (up to sign) a sum of Temperley-Lieb immanants by Proposition \ref{prop:cm equals imm sum}. This yields a long expression for $\imm_w^{\%}$ in terms of Temperley-Lieb immanants. We then show that almost all of the terms cancel, leaving us with a single Temperley-Lieb immanant $\imm_w$, up to sign. The exact sign can then be extracted by comparing the coefficients of $x_w.$
\par In order to do this, we first proceed by classifying the permutations that avoid the three patterns $321, 2143, 1324$. We begin with the following proposition:
\begin{prop}\label{321-avoiding Rectangles}
If $w$ is a $321$-avoiding permutation, then the complement of the diagram of $\imm^{\%}_w$ consists of two (possibly empty) rectangles in the corners: the rectangle in the upper-left corner is $(w^{-1}(1) - 1)$ by $(w(1) - 1)$, and the one in the lower-right corner is $(n - w^{-1}(n))$ by $(n - w(n))$. In particular, $m_w(i)$ and $M_w(i)$ takes on at most two distinct values across all $1 \le i \le n$. (Recall that $m_w$ and $M_w$ were defined in Definition \ref{defn:per imm w}.)
\end{prop}
\begin{proof}
Notice that if $m_w(i)$ changes values at three times, say at $i_1, i_2, i_3,$ then we have $w(i_1) > w(i_2) > w(i_3),$ which is a contradiction to the assumption that $w$ is $321$-avoiding. Similarly, we see that if $M_w$ changes three times, at $i_1, i_2, i_3,$ we have that $M_w(i_1) = w(i_1) > M_w(i_2) = w(i_2) > M_w(i_3) = w(i_3),$ again a contradiction. Thus, both $m_w (i)$ and $M_w (i)$ take on at most two distinct values across all $1 \le i \le n$. This means the diagram of $\imm_w^\%$ is the complement of two rectangles in the upper-left and lower-right corners. Since $m_w (1) = w(1)$ and $m_w (w^{-1} (1)) = 1$, the dimensions of the upper-left rectangle are $(w^{-1}(1) - 1)$ by $(w(1) - 1)$. Similarly, the rectangle in the lower-right corner has dimensions $(n - w^{-1}(n))$ by $(n - w(n))$. This completes the proof of the Proposition.
\end{proof}
\par Before proceeding with the proof of the proposition, we utilize symmetries of Temperley-Lieb immanants and \%-immanants in order to make some assumptions about our permutation $w.$ The following lemma formalizes the symmetries of these immanants.

\begin{lemma}\label{lem: simplify cases}
Let $S$ be the linear map that sends $x_{\sigma}$ to $x_{\sigma^{-1}},$ and let $T$ be the linear map that sends $x_{\sigma}$ to $x_{w_0 \sigma w_0},$ where $w_0$ is the longest word in $\fkS_n.$ Then, $S$ sends $\imm_w$ to $\imm_{w^{-1}}$ and $\imm^{\%}_w$ to $\imm^{\%}_{w^{-1}},$ and $T$ sends  $\imm_w$ to $\imm_{w_0ww_0}$ and $\imm^{\%}_w$ to $\imm^{\%}_{w_0ww_0}.$
\end{lemma}
\begin{remark}
The proof of this lemma, found in subsection \ref{subsec: sec4proof}, can be easily extended to show that for any skew-tableau $\lambda/\mu$, we have $S(\imm_{\lambda/\mu})$ equals the \%-immanant of the reflection of $\lambda/\mu$ across the main diagonal of the $n \times n$ bounding box, while $T(\imm_{\lambda/\mu})$ equals the \%-immanant of the $180^{\circ}$ rotation of $\lambda/\mu$ about the center of the bounding box.
\end{remark}
Now, we start the proof of Theorem \ref{CSBThm}. Our first step is to use Lemma \ref{lem: simplify cases} to reduce Theorem \ref{CSBThm} to a special case.
\begin{prop}\label{prop:CSBreduction}
    If Theorem \ref{CSBThm} is true for all $321$-, $2143$-, and $1324$-avoiding permutations $w \in \fkS_n$ such that either $w(1) = 1$ or $w(1) = w(n) + 1$, then Theorem \ref{CSBThm} is true for all $321$-, $2143$-, and $1324$-avoiding permutations $w \in \fkS_n$.
\end{prop}

\begin{proof}
    \par Fix a $321$-, $2143$-, and $1324$-avoiding permutation $w \in \fkS_n$. Our goal is to show Theorem \ref{CSBThm} is true for $w$. We divide our analysis according to the following four cases. 
\begin{enumerate}
    \item $w(1) = 1$.
    \item $w(1) \neq 1, w(n) = n.$
    \item $w(1) \neq 1, w(n) \neq n, w(1)> w(n)$
    \item $w(1) \neq 1, w(n) \neq n, w(1) < w(n)$

\end{enumerate}
\par Case 1: Since $w(1) = 1$, Theorem \ref{CSBThm} is true for $w$.
\par Case 2: Consider $w'=w_0ww_0$.
By Corollary \ref{cor: w_0 conjugation pattern avoidance}, $w'$ is $321$-, $2143$-, and $1324$-avoiding if and only if $w$ is.
Furthermore, by Lemma \ref{lem: simplify cases}, $T$ sends $\imm_w$ to $\imm_{w'}$ and $\imm^{\%}_w$ to $\imm^{\%}_{w'}.$ So $\imm^{\%}_w=\imm_w$ if and only if $\imm^{\%}_{w'}=\imm_{w'}$. Thus, we have reduced showing Theorem \ref{CSBThm} holds for $w$ to showing it holds for $w'$. But now, $w'(1)=n+1-w(n)= 1$, so Theorem \ref{CSBThm} is true for $w'$.
\par Case 3: First suppose that $w(1) > w(n) + 1$. If $c = w^{-1} (w(1) - 1)$, then $1 < c < n$ forms a $321$-pattern because $w(1) > w(c) > w(n)$. This is a contradiction. So we must have $w(1) = w(n) + 1$, which means Theorem \ref{CSBThm} is true for $w$.
\par Case 4: Let $a = w^{-1} (1)$ and $b = w^{-1} (n)$; we claim that $a > b$. Suppose not; then the fact $a < b$ together with the assumption of Case 4 gives $1 < a < b < n$ and $w(a)=1<w(1)<w(n)<n=w(b)$. So $1<a<b<n$ forms a $2143$-pattern, contradicting with our assumption that $w$ avoids $2143$. Thus, we must have that $w^{-1} (1)>w^{-1} (n)$. For simplicity of notation, let $w'=w^{-1}$.
By Lemma \ref{lem: restriction inverses}, $w'$ is $321$-, $2143$-, and $1324$-avoiding if and only if $w$ is.
Furthermore, by Lemma \ref{lem: simplify cases}, $S$ sends $\imm_w$ to $\imm_{w'}$ and $\imm^{\%}_w$ to $\imm^{\%}_{w'}.$ So $\imm^{\%}_w=\imm_w$ if and only if $\imm^{\%}_{w'}=\imm_{w'}$. Thus, we have reduced showing Theorem \ref{CSBThm} holds for $w$ to showing it holds for $w'$. But since $w^{-1} (1)>w^{-1} (n)$, we can use the analysis of Case 3 to conclude that Theorem \ref{CSBThm} holds for $w'$. This completes the proof of Proposition \ref{prop:CSBreduction}
\end{proof}

\par By Proposition \ref{prop:CSBreduction}, we can assume that either $w(1) = 1$ or $w(1) = w(n) + 1$. We will assume this condition holds throughout the rest of this subsection. 
\par To begin with, we have the following proposition, which expresses $\imm_w^\%$ as a sum of certain products of complementary minors.
\begin{prop}\label{Rectangle Complementary Minors}
Suppose that $w\in \mathfrak{S}_n$ is a $321$-, $1324$-, $2143$-
avoiding permutation so that either $w(1) = 1$ or $w(1) = w(n) + 1$. Then we have \[\imm_w^{\%} = \sum\limits_{\substack{|I| = w(n) \\ [w^{-1}(n) + 1, n] \subset I \subset [w^{-1}(1), n]}} \cm_{I, [1, w(n)]}.\]
\end{prop}
As an example we have the following sum when $w = 3142,$ with the shaded regions representing which complementary minors we take.
\begin{center}
    \includegraphics[]{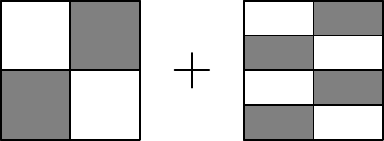}
\end{center}
\begin{proof}
\par We first observe the following identity which holds in general. Let $X$ be the determinant of the matrix whose $(i, j)$ entry is the variable $x_{ij}.$ Then, for any $1 \le k \le n,$ 
\begin{equation}\label{eqn:ebm}
    X = \sum_{|I| = k} \cm_{I,[1,k]}.
\end{equation}
(The case $k = 1$ is expansion by minors along the first column.) To prove this, we compare coefficients of $x_\sigma$ on both sides. The coefficient of $x_\sigma$ in $X$ is $\sgn(\sigma)$. For the RHS, there is a unique $I$ for which $\cm_{I,[1,k]}$ has nonzero coefficient for $x_\sigma$: namely when $I = \sigma^{-1} ([1,k])$. Then, by definition of $\cm_{I,J}$, the coefficient of $x_\sigma$ in $\cm_{\sigma^{-1} ([1,k]), [1,k]}$ is $\sgn(\sigma)$. Thus, the coefficients of $x_\sigma$ on both sides of equation \ref{eqn:ebm} are equal, proving the identity.

\par From Proposition \ref{321-avoiding Rectangles}, we have that $\imm_w^{\%}$ is obtained from $X$ by setting $x_{ij} = 0$ for $1 \le i \le w^{-1} (1)-1, 1 \le j \le w(1)-1$ and $w^{-1} (n) + 1 \le i \le n, w(n) \le j \le n$.

\par Now, consider equation \ref{eqn:ebm} with $k = w(n)$, and plug in $x_{ij} = 0$ for $1 \le i \le w^{-1} (1)-1, 1 \le j \le w(1)-1$ and $w^{-1} (n) + 1 \le i \le n, w(n) \le j \le n$. We will show that $\cm_{I,[1,w(n)]}$ vanishes unless $[w^{-1}(n) + 1, n] \subset I \subset [w^{-1}(1), n]$.

\par First, if $i \notin I$ for some $w^{-1} (n) + 1 \le i \le n$, then the $i$-th row in the matrix for $\cm_{I,[1,w(n)]}$ will be all zeros, so $\cm_{I,[1,w(n)]} = 0$. Thus, $\cm_{I,[1,w(n)]}$ vanishes unless $[w^{-1}(n) + 1, n] \subset I$.

\par Next, if $w^{-1} (1) = 1$, then $I \subset [w^{-1}(1), n]$ is vacuously true. If $w^{-1} (1) > 1$, then by assumption we have $w(1) = w(n) + 1.$ Thus, $x_{ij} = 0$ for $1 \le i \le w^{-1} (1) - 1$, $1 \le j \le w(n)$. Now, if $i \in I$ for some $1 \le i \le w^{-1} (1) - 1$, then the $i$-th row of the matrix for $\cm_{I,[1,w(n)]}$ will be all zeros, so $\cm_{I,[1,w(n)]} = 0$. Thus, in either case, $\cm_{I,[1,w(n)]}$ vanishes unless $I \subset [w^{-1}(1), n]$. Hence, we have showed $\cm_{I,[1,w(n)]}$ vanishes unless $[w^{-1}(n) + 1, n] \subset I \subset [w^{-1}(1), n]$.

Finally, if $[w^{-1}(n) + 1, n] \subset I \subset [w^{-1}(1), n]$, then $\cm_{I,[1,w(n)]}$ remains unchanged after setting some $x_{ij}$ to zero. This proves the Proposition.
\end{proof}
From here, we can now express our complementary minors as sums of Temperley-Lieb immanants by Proposition \ref{prop:cm equals imm sum}. Before doing this, we first observe that the pattern avoidance conditions let us conclude that $w$ takes the following form. 
\begin{lemma}\label{lem: w form}
Suppose that $w$ is a $321$-, $1324$-, $2143$- avoiding permutation, and furthermore either $w(1) = 1$ or $w(1) = w(n) + 1$. Then, $w$ is uniquely determined and has one line notation $$(w(n)+1..w^{-1}(1)+w(n)- 1:1..w^{-1}(n) + w(n) - n:w^{-1}(1) + w(n)..n:w^{-1}(n) + w(n) - n+1..w(n)).$$  
\end{lemma}
\begin{remark}
Specifically, $w$ has block structure $[3][1][4][2]$ with block lengths $w^{-1} (1) - 1, w^{-1} (n) + w(n) - n, n - w^{-1} (1) - w(n) + 1$, and $n - w^{-1} (n)$; some of the blocks may be empty.
\end{remark}

We also need the following lemma about colorings.

\begin{lemma}\label{lem:simple-no-internal-pair}
Suppose $a, b, c$ satisfy $0 \leq a, b, c \leq n$ and $a + b + c = 2n$. There is a unique coloring of the vertices $1, 2, \cdots, 2n$ and a unique compatible non-crossing matching such that vertices $[a+1, a+b]$ are colored black, vertices $[a+b+1, 2n]$ are colored white, and there does not exist an internal pairing within $[1, a]$ (i.e. there do not exist $1 \le i < j \le a$ that are paired).
\end{lemma}

With these two lemmas, we will argue in the next proposition that our sum of complementary minors is (up to sign) a single Temperley-Lieb immanant $\imm_w$.

\begin{prop}\label{Rectangle Unique Matching}
Suppose that $w\in \mathfrak{S}_n$ is a $321$-, $1324$-, $2143$-
avoiding permutation so that either $w(1) = 1$ and $w(n) \neq n,$ or $w(1) = w(n) + 1$. Then there exists $c_w \in \R$ such that \[\sum\limits_{|I| = w(n), [w^{-1}(n) + 1, n] \subset I \subset [w^{-1}(1), n]} \cm_{I, [1, w(n)]} = c_w \imm_w.\]
\end{prop}
\begin{proof}
By Lemma \ref{lem:simple-no-internal-pair} and Proposition \ref{prop:ncm}, there exists a unique $\tw$ and a unique coloring $\tilde{\cC}$ compatible with $\ncm(\tw)$ such that $[1, \tw^{-1} (1) - 1]$ and $[1, \tw(n)]'$ are colored white, $[\tw^{-1} (n)+1, n]$ and $[\tw(n)+1, n]'$ are colored black, and there are no internal pairings between vertices in $[\tw^{-1} (1), \tw^{-1} (n)]$.

We first show $\tw = w$. To do this, it suffices to show that $w$ and the following coloring $\cC$
satisfy the conditions for $(\tw, \tilde{\cC})$:
\begin{itemize}
    \item $[1, w^{-1} (1) - 1]$, $[1, w(n)]'$, and $[w^{-1} (1) + w^{-1} (n) + w(n) - n, w^{-1} (n)]$ are colored white;

    \item $[w^{-1} (n)+1, n]$, $[w(n)+1, n]'$, and $[w^{-1} (1), w^{-1} (1) + w^{-1} (n) + w(n) - n - 1]$ are colored black.
\end{itemize}
First, by Lemma \ref{lem: w form}, the coloring $\cC$ satisfies the conditions of Lemma \ref{lem:matching_compatibility}, which means that $\ncm(w)$ and $\cC$ are compatible. Next, we show that there are no internal pairings between vertices in $[w^{-1} (1), w^{-1} (n)]$. There are no internal pairings within $[w^{-1} (1) + w^{-1} (n) + w(n) - n, w^{-1} (n)]$ or $[w^{-1} (1), w^{-1} (1) + w^{-1} (n) + w(n) - n - 1]$ because of color, and there are no pairings between $[w^{-1} (1) + w^{-1} (n) + w(n) - n, w^{-1} (n)]$ and $[w^{-1} (1), w^{-1} (1) + w^{-1} (n) + w(n) - n - 1]$ because by Lemma \ref{lem:matching_compatibility}, any $i \in [w^{-1} (1), w^{-1} (1) + w^{-1} (n) + w(n) - n - 1]$ must be paired with some $j$ or $j'$ with $j \le i$. This shows $(w, \cC)$ satisfy the conditions for $(\tw, \tilde{\cC})$, so $\tw = w$.

Now, we use Proposition \ref{prop:cm equals imm sum} and Lemma \ref{cmminors} to expand the sum of complementary minors in Proposition \ref{Rectangle Complementary Minors} as $\sum_{v} c_v \imm_v$ for some constants $c_v$. We claim that if $v \neq w$, then $c_v = 0$. Suppose not; then since $c_v \neq 0$, there exists a coloring compatible with $\ncm(v)$ such that $[1, v^{-1} (1) - 1]$ and $[1, v(n)]'$ are colored white, $[v^{-1} (n)+1, n]$ and $[v(n)+1, n]'$ are colored black. Since $v \neq \tw$, $\ncm(v)$ must have an internal pairing between two vertices $p, q \in [v^{-1} (1), v^{-1} (n)]$. Since $p, q$ are paired in $\ncm(v)$, we see by Remark 2 after Definition \ref{def:ncm} that $q-p$ must be odd.

For an set $I$ satisfying $[w^{-1} (n)+1, n] \subset I \subset [w^{-1} (1), n]$, define the involution $\iota$ by sending $I$ to $(I - \{p\}) \cup \{q\}$ if $p \in I,$ and $(I - \{q\}) \cup \{p\}$ otherwise. Then, by Lemma \ref{cmminors} and the fact that $q-p$ is odd, the coefficients of $\imm_v$ in $\cm_{I, [1,w(n)]}$ and $\cm_{\iota(I), [1,w(n)]}$ are negatives of each other. Thus, after cancellation, we see that $c_v = 0$ whenever $v \neq w$.
\end{proof}

Putting all of the pieces together yields Theorem \ref{CSBThm}.

\begin{proof}[Proof of Theorem \ref{CSBThm}]
Using Proposition \ref{prop:CSBreduction}, we may assume that either $w(1) = 1$ or $w(1) = w(n) + 1$. Then by Propositions \ref{Rectangle Complementary Minors} and \ref{Rectangle Unique Matching}, we obtain $\imm^\%_w = c_w \imm_w$. Comparing coefficients of $x_w$ on both sides, we have $\sgn(w) = c_w f_w (w)$. But $f_w (w) = 1$ by Lemma \ref{lem:basic fwu}, so $c_w = \sgn(w)$. This proves the theorem.
\end{proof}

\subsection{$\imm_w = \sgn(w)\imm_w^{\%}$ Implies 2143-, 1324-Avoiding}
In this subsection, we prove the only if direction of Theorem \ref{thm:onepercentimm}.
\begin{thm}\label{2143ConverseThm}
If $w$ is a permutation that avoids $321$ such that $\imm_w$ is a \%-immanant up to sign, then $w$ must avoid both $1324$ and $2143$. 
\end{thm}
Our strategy is to first show $w$ avoids $1324$. Then, we show that if $w$ avoids $1324$, then $w$ must also avoid $2143$. 
\begin{thm}\label{1324Thm}
If $w$ is a $321$-avoiding permutation where $\imm_w$ is a linear combination of \%-immanants, then $w$ must be $1324$-avoiding.
\end{thm}
\begin{proof}
Assume for contradiction that $w$ contains the pattern $1324$. Then we have indices $i<j<k<l$ such that $w(i)<w(k)<w(j)<w(l)$. Define the permutation $w':=w \cdot (j,k)$. By Lemma \ref{cor:bruhat_inv}, $w' \le w$. By Lemma \ref{lem:basic fwu}, we have $f_{w}(w')=0$ and $f_w(w)=1$.  

However, by Theorem \ref{thm:classifying space of percent}, we need $f_w(w') = -f_w(w),$ which is a contradiction, as desired.
\end{proof}


Now, to finish proving Theorem \ref{2143ConverseThm}, we need to show that $w$ avoids $2143$. We begin by analyzing what happens if $w$ doesn't avoid $2143$.

\begin{lemma}\label{w(1) and w(n) inequality for 2143}
Suppose that $w$ contains $2143$.
\begin{enumerate}[(a)]
    \item If $w$ avoids $321$, then $w(1) < w(n)$ and $w^{-1}(1) < w^{-1}(n)$.
    
    \item If $w$ avoids $321$, then $w^{-1}(1) + w(1) \le n + 1$ and $(n+1 - w^{-1}(n)) + (n+1 - w(n)) \le n+1$.
    
    \item If $w$ avoids $1324$, then $w(1) \neq 1$ and $w(n) \neq n.$
\end{enumerate}
\end{lemma}

Furthermore, if we are given that $\imm_w$ is a \%-immanant up to sign, we would like to know some properties of the \%-immanant, beyond the one given in Lemma \ref{lem:bigtableau}.

\begin{lemma}\label{lem:corner zero}
Suppose $w$ avoids $321, 1324$ and contains $2143$, and $\imm_w$ is a \%-immanant $\imm^\%_{\lambda/\mu}$ up to sign. Then:
\begin{enumerate}[(a)]
    \item $(1, 1), (n, n) \notin \lambda/\mu$;
    
    \item There exists $i$ such that $(1, i), (n+1-i, 1), (n, i), (n+1-i, n) \in \lambda/\mu$.
\end{enumerate}
\end{lemma}

These lemmas, combined with Theorem \ref{1324Thm}, allow us to prove Theorem \ref{2143ConverseThm}.
\begin{proof}[Proof of Theorem \ref{2143ConverseThm}]
Suppose that $\imm_w$ is a \%-immanant $\imm^\%_{\lambda/\mu}$, up to a sign. By Theorem \ref{1324Thm}, $w$ is $1324$-avoiding. Assume for the sake of contradiction that $w$ is not 2143-avoiding. Then, from Lemma \ref{lem:corner zero}, we can find $i'$ such that $(1, i'), (n+1-i', 1), (n, i'), (n+1-i', n) \in \lambda/\mu$ and $(1, 1), (n, n) \notin \lambda/\mu$. 
\par Our approach now is to construct an $n \times n$ matrix $X,$ and then show that $\imm_w$ and $\imm^\%_{\lambda/\mu}$ give different results when evaluated on $X$. Let $X = (x_{i,j})$ where $x_{i, n+1 - i} = 1$ for all $i,$ and also set $x_{1, n + 1 - i'} = x_{1, 1} = x_{i', 1} = x_{i', n} = x_{n, n} = x_{n, n + 1 - i'} = 1.$ Set the rest of the entries to be zero.
\par We now evaluate $\imm^\%_{\lambda/\mu} (X)$. Notice that, by our assumption, this is equal to the determinant of the matrix $X'$ by setting $x_{1, 1}$ and $x_{n, n}$ to zero in $X,$ by Lemmas \ref{lem:bigtableau} and \ref{lem:corner zero}. In other words, we are evaluating the determinant of the following matrix: $$\begin{pmatrix} 0 & 0 & \hdots & 1 & \hdots & 0 & 1 \\ 0 & 0 & \hdots & 0 & \hdots & 1 & 0 \\ \vdots & \vdots & \ddots & \ddots & \ddots & \vdots & \vdots \\ 1 & 0 & \hdots & 1 & \hdots & 0 & 1 \\ \vdots & \vdots & \ddots & \ddots & \ddots & \vdots & \vdots \\ 0 & 1 & \hdots & 0 & \hdots & 0 & 0 \\ 1 & 0 & \hdots & 1 & \hdots & 0 & 0 \end{pmatrix}.$$
\par Notice that we can re-arrange these rows (which only affects the sign) so that the $i'$-th row and the $n$th row become the 2nd and 3rd row, respectively. Then, re-arrange the columns so that the $(n+1-i')$-th column and the $n$th column become the 2nd and 3rd columns. The resulting determinant we are evaluating is then $$\begin{pmatrix} 0 & 1 & 1 & 0 & \hdots & 0 & 0 & 0\\ 1 & 1 & 1 & 0 & \hdots & 0 & 0 & 0 \\ 1 & 1 & 0 & 0 & \hdots & 0 & 0 & 0 \\  0 & 0 & 0 & \ddots & 0 & 0 & 0 & 1  \\  0 & 0 & 0 & \hdots & 0 & 0 & 1 & 0 \\ \vdots & \vdots & \vdots & \ddots & \ddots & \vdots & \vdots & \vdots \\ 0 & 0 & 0 & 0 & 0 & 1 & \hdots & 0 \\ 0 & 0 & 0 & 0 & 1 & 0 & \hdots & 0 \\ 0 & 0 & 0 & 1 & 0 & 0 & \hdots & 0 \end{pmatrix}.$$
\par Notice that this is a block-diagonal matrix, with a matrix with all $1$s along the anti-diagonal in one block and the matrix $\begin{pmatrix}0 & 1 & 1 \\ 1 & 1 & 1 \\ 1& 1 &  0 \end{pmatrix}$ as the other. This means that the determinant of this matrix in total is equal to simply the product of the determinants of these two blocks, which is just $1.$ There are some sign considerations, but this is not important. We conclude $\imm^\%_{\lambda/\mu} (X) = \pm 1$.
\par On the other hand, since $X$ has three equal rows (rows $1, i$ and $n$), Proposition 3.14 from \cite{RS} tells us that $\imm_w X = 0.$ In particular, for our $w,$ we know that $\imm_w X = 0.$ 
But this contradicts $\imm^\%_{\lambda/\mu} (X) = \pm 1$. This proves the theorem.
\end{proof}
\subsection{Proofs of Lemmas in Section 4}\label{subsec: sec4proof}

In this subsection we present the proofs of the various lemmas throughout this section. We also reproduce each of the lemmas in this subsection.



\begin{replemma}{lem: simplify cases}
Let $S$ be the linear map that sends $x_{\sigma}$ to $x_{\sigma^{-1}},$ and let $T$ be the linear map that sends $x_{\sigma}$ to $x_{w_0 \sigma w_0},$ where $w_0$ is the longest word in $\fkS_n.$ Then, $S$ sends $\imm_w$ to $\imm_{w^{-1}}$ and $\imm^{\%}_w$ to $\imm^{\%}_{w^{-1}},$ and $T$ sends  $\imm_w$ to $\imm_{w_0ww_0}$ and $\imm^{\%}_w$ to $\imm^{\%}_{w_0ww_0}.$
\end{replemma}
\begin{proof}
The claim about $\imm_w$ is Lemma \ref{lem:symmetry}, so we will show the Lemma for $\imm^{\%}_w$.

Note that the coefficient of $x_{\sigma}$ in $\imm^{\%}_{w^{-1}}$ is $\sgn(\sigma)$ if and only if $m_{w^{-1}}(i) \leq \sigma(i) \leq M_{w^{-1}}(i)$ for each $i,$ and is $0$ otherwise. But now, consider $\sigma^{-1}(i).$ Notice that $m_{w^{-1}}(i) \leq \sigma(i)$ implies that $w^{-1}(k) = l \leq \sigma(i)$ for some $k \leq i.$ Furthermore, we know that $w(l) \leq k \leq i,$ meaning that $m_w(\sigma(i)) \leq i$ for each $i,$ or that $m_w(i) \leq \sigma^{-1}(i).$ Similarly, $w^{-1}(p) = q \geq \sigma(i)$ for some $p \geq i,$ so then $w(q) = p \geq i,$ and thus $M_w(\sigma(i)) \geq i$ for each $i.$ We combine these results to get that $m_w(i) \leq \sigma^{-1}(i) \leq M_w(i).$ In particular, $x_{\sigma}$ has coefficient $\sgn(\sigma)$ in $\imm_w^{\%}$ if and only if $x_{\sigma^{-1}}$ has coefficient $\sgn(\sigma^{-1})$ in $\imm_{w^{-1}}^{\%}.$ Thus, $S$ sends $\imm^{\%}_{w^{-1}}$ to $\imm^{\%}_w.$ It's not hard to see that $S$ is an involution, so the reverse also holds. Similarly for $T,$ note that $m_w(i) \leq \sigma(i) \leq M_w(i)$ for each $i$ if and only if $\max \{n + 1 - w(1),  n + 1 - w(2), \ldots,  n + 1 - w(n + 1 - i)\} = M_{w_0ww_0}(i) \geq n + 1 - \sigma(n + 1 - i) \geq \min \{n + 1 - w(n + 1 - i), n + 1 - w(n + 2 - i), \ldots, n + 1 - w(n)\} =  m_{w_0ww_0}(i).$ In other words, $x_{\sigma}$'s coefficient in $\imm_w^{\%}$ is the same as $x_{w_0\sigma w_0}$'s coefficient in $\imm_{w_0ww_0}^{\%}$ for each $u,$ meaning that $T$ sends the former to the latter.
\end{proof}
\begin{replemma}{lem: w form}
Suppose that $w$ is a $321$-, $1324$-, $2143$- avoiding permutation, and furthermore either $w(1) = 1$ or $w(1) = w(n) + 1$. Then, $w$ is uniquely determined and has one line notation $$(w(n)+1..w^{-1}(1)+w(n)- 1:1..w^{-1}(n) + w(n) - n:w^{-1}(1) + w(n)..n:w^{-1}(n) + w(n) - n+1..w(n)).$$
\end{replemma}
\begin{proof}
We first show $w^{-1}$ is increasing from $1$ to $w(n).$ Otherwise, if $w^{-1}(i) > w^{-1}(j)$ but $1 \leq i < j \leq w(n),$ then either $w(1) \neq 1,$ and so $1 < w^{-1}(j) < w^{-1}(i)$ forms a $321$-pattern (note that $w(1) = w(n) + 1 > j > i$), or $w(1) = 1,$ where $i, j \neq 1$ yields that 
$1 < w^{-1}(j) < w^{-1}(i) < n$ forms a $1324$ pattern. Similarly, $w^{-1}$ is increasing from $w(n) + 1$ to $n.$
\par We first show $w(1..w^{-1} (1) - 1) = (w(n)+1..w(n) + w^{-1} (1) - 1)$. If $w(1) = 1$ then this statement is trivial. If $w(1) > 1$ then $w(1) = w(n) + 1$, and so by $321$-avoidance of $w$ (or Proposition \ref{321-avoiding Rectangles}), we have $w(i) \ge w(n) + 1$ for $1 \le i \le w^{-1} (1) - 1$. Since $w^{-1}$ is increasing from $w(n)+1$ to $n$, we conclude that $w(1..w^{-1} (1) - 1) = (w(n)+1..w(n) + w^{-1} (1) - 2)$.
\par Next, we show that $w(w^{-1}(n) + 1..n) = (w(n) + w^{-1} (n) - n+1..w(n))$. By $321$-avoidance of $w$, we know that $w(i) \le w(n)$ for $w^{-1}(n) + 1 \le i \le n$. Since $w^{-1}$ is increasing from $1$ to $w(n)$, we conclude that $w(w^{-1}(n) + 1..n) = (w(n) + w^{-1} (n) - n+1..w(n))$.
\par As for the remaining values for $w$, notice that since $w$ avoids $1324$, $w$ needs to be increasing from $w^{-1}(1)$ to $w^{-1}(n)$. We also know
\begin{equation*}
    w([w^{-1} (1), w^{-1} (n)]) = [n] \setminus w([1, w^{-1} (1) - 1:w^{-1}(n) + 1, n]) = [1, w(n)+w^{-1}(n)-n : w(n)+w^{-1} (1), n].
\end{equation*}
Thus, we get $w(w^{-1} (1)..w^{-1} (1)+w(n)+w^{-1}(n)-n-1) = (1..w(n)+w^{-1}(n)-n)$ and $w(w^{-1} (1)+w(n)+w^{-1}(n)-n..w^{-1} (n)) = (w(n)+w^{-1} (1)..n)$, which corresponds to the one line notation 
$$(w(n)+1..w^{-1}(1)+w(n)- 1:1..w^{-1}(n) + w(n) - n:w^{-1}(1) + w(n)..n:w^{-1}(n) + w(n) - n+1..w(n)),$$
proving the lemma.
\end{proof}

\begin{replemma}{lem:simple-no-internal-pair}
Suppose that $a, b, c$ satisfy $0 \le a,b,c \le n$ and $a + b + c = 2n.$ There is a unique coloring of the vertices in $[2n]$ and a unique compatible non-crossing matching such that vertices $[a+1, a+b]$ are colored black, vertices $[a+b+1, 2n]$ are colored white, and there do not exist internal pairings within $[1, a]$ (i.e. there do not exist $1 \le i < j \le a$ that are paired).
\end{replemma}

\begin{proof}
We induct on $a$. For the base case $a = 0$, we have $b = c$ and the unique matching is given by pairing $x$ and $2n-x$ for all $1 \le x \le n$. For the inductive step, suppose the result is true for $a-1$. We may assume $b \le c$; otherwise, relabel the vertices $x \to a+1 - x \pmod {2n}$ and flip black/white. Then $a \ge 1$ implies $2b \le b+c \le 2n-1$, so $b \le n-1$.

We claim that $1$ is paired with $2n$. Suppose not; then $1$ is paired with $x$ for some $a+1 \le x \le 2n-1$. Among the vertices in $[x+1, 2n]$, there must be equal numbers of black vertices and white vertices. We now show this can't be the case. If $x \ge a+b$, then all the vertices are black. If $a+1 \le x < a+b$, then there are $c$ white vertices and at most $b-1$ black vertices. Thus, in fact $1$ must be paired with $2n$, and $1$ must be colored black. If we remove $1$ and $2n$, then the remaining configuration exhibits a coloring and compatible non-crossing matching. By the inductive hypothesis (note that $0 \le a-1, b, c-1 \le n-1$), this coloring and compatible non-crossing matching are unique.
\end{proof}
\begin{replemma}{w(1) and w(n) inequality for 2143}
Suppose that $w$ contains $2143$.
\begin{enumerate}[(a)]
    \item If $w$ avoids $321$, then $w(1) < w(n)$ and $w^{-1}(1) < w^{-1}(n)$.
    
    \item If $w$ avoids $321$, then $w^{-1}(1) + w(1) \le n + 1$ and $(n+1 - w^{-1}(n)) + (n+1 - w(n)) \le n+1$.
    
    \item If $w$ avoids $1324$, then $w(1) \neq 1$ and $w(n) \neq n.$
\end{enumerate}
\end{replemma}
\begin{proof}
(a) We first show that $w(1) < w(n)$. Since it contains a $2143$ pattern, we can find $a < b < c < d$ such that $w(b) < w(a) < w(d) < w(c)$. Then $w(1) \le w(a)$ (otherwise $1, a, b$ would be a $321$-pattern) and $w(n) \ge w(d)$ (otherwise $c, d, n$ would be a $321$-pattern), so $w(1) \le w(a) < w(d) \le w(n)$.

Now $w^{-1} (1) < w^{-1} (n)$ follows from applying the first part of (a) to $w^{-1}$, which also avoids $321$ and contains $2143$ by Lemma \ref{lem: restriction inverses}.

(b) Note that $w(i) > w(1)$ for $2 \le i \le w^{-1} (1) - 1$ and $i = n$ by $321$-avoidance and part (a). But there are exactly $n-w(1)$ many $i$ with $w(i) > w(1)$, which shows that $w^{-1} (1) - 1 \le n - w(1)$ and hence $w(1) + w^{-1} (1) \le n+1.$

Then $(n+1 - w^{-1}(n)) + (n+1 - w(n)) \le n+1$ follows from applying the first part of (b) to $w_0 w w_0$, which also avoids $321$ and contains $2143$ by Corollary \ref{cor: w_0 conjugation pattern avoidance}.

(c) We first show $w(1) \neq 1$. If $w(1) = 1$, then there exist $1 < a < b < c < d$ such that $w(b) < w(a) < w(d) < w(c)$. Then $1 < a < b < c$ satisfy $w(1) < w(b) < w(a) < w(c)$, which contradicts $w$ being $1324$ avoiding. Thus, $w(1) \neq 1$, and $w(n) \neq n$ follows by applying the first part of (c) to $w_0 w w_0$, which also avoids $1324$ and contains $2143$ by Corollary \ref{cor: w_0 conjugation pattern avoidance}.
\end{proof}

\begin{replemma}{lem:corner zero}
Suppose $w$ avoids $321, 1324$ and contains $2143$, and $\imm_w$ is a \%-immanant $\imm^\%_{\lambda/\mu}$ up to sign. Then:
\begin{enumerate}[(a)]
    \item $(1, 1), (n, n) \notin \lambda/\mu,$
    
    \item There exists $i$ such that $(1, i), (n+1-i, 1), (n, i), (n+1-i, n) \in \lambda/\mu$.
\end{enumerate}
\end{replemma}

\begin{proof}
(a) By Lemma \ref{w(1) and w(n) inequality for 2143}(c), we have $w(1) \neq 1$. If $(1, 1) \in \lambda/\mu$, then consider $w' = w \cdot (1, w^{-1} (1))$. Then since $(1, w^{-1} (1))$ is an inversion of $w'$, we have $w' \not\ge w$ by Corollary \ref{cor:bruhat_inv}. Thus, by Lemma \ref{lem:basic fwu}, we have $f_w (w') = 0$ and $f_w (w) = 1$. However, note that if the coefficient of $x_w$ is nonzero, then in particular we require that $n \leq \lambda_{w^{-1}(n)}$ and $\mu$ to be empty, since $\mu_1 = 0$ by $\lambda/\mu$ containing $(1, 1).$ By Lemma \ref{w(1) and w(n) inequality for 2143}(a), we have that $w^{-1}(1) < w^{-1}(n),$ and so $\lambda_{w^{-1}(1)} = n.$ But then the coefficients of $x_w$ and $x_{w'}$ in $\imm^\%_{\lambda/\mu}$ are negatives of each other, since $w'(i) \leq \lambda_i$ for $i \in [n]$ holds if and only if $w(i) \leq \lambda_i$ for $i\in [n].$ Thus, $(1, 1) \notin \lambda/\mu$, and an analogous argument shows $(n, n) \notin \lambda/\mu$.

(b) Choose $i = \max(w(1), n+1-w^{-1}(n))$. We first observe that since $f_w (w) = 1$, we must have $(j, w(j)) \in \lambda/\mu$ for all $1 \le j \le n$. Thus, since $i \ge w(1)$, we must have $(1, i) \in \lambda/\mu$. Similarly, since $n+1-i \le w^{-1} (n)$, we must have $(n+1-i, n) \in \lambda/\mu$.

\par Now, for $(n, i),$ observe that either $i = w(1) < w(n),$ or $i = n + 1 - w^{-1}(n) \leq w(n)$ by Lemma \ref{w(1) and w(n) inequality for 2143}(a,b). In either case we see that $(n, i) \in \lambda/\mu.$ Finally, we observe that in $(n + 1 - i, 1),$ either $n + 1 - i = w^{-1}(n) > w^{-1}(1),$ or $n + 1 - i = n + 1 - w(1) \geq w^{-1}(1),$ so again in either case $(n+1 - i, 1) \in \lambda/\mu.$ This proves the claim.
\end{proof}

\section{Temperley-Lieb Immanants as Linear Combinations of \%-Immanants}\label{sec:general_TL}
This section is devoted to proving the following theorem. 
\begin{thm}\label{TwoPercentForward}
Let $w$ be a $321$-avoiding permutation. The following statements are equivalent:
\begin{enumerate}
    \item The Temperley-Lieb immanant $\imm_w$ is a linear combination of \%-immanants;
    
    \item The signed Temperley-Lieb immanant $\sgn(w) \imm_w$ is a sum of at most two \%-immanants;
    
    \item The permutation $w$ avoids the patterns $1324, 24153, 31524, 231564$, and $312645,$ in addition to avoiding $321.$
\end{enumerate}
\end{thm}

To prove the theorem, we first classify permutations $w$ that avoid $321$ and $1324$ but contain $2143.$ Then, as in Section \ref{sec:specific-TL}, we will use Proposition \ref{prop:cm equals imm sum} to express $\imm_w$ as an explicit linear combination of certain complementary minors. Finally, we compute the coefficients $f_w (u)$ of the immanant $\imm_w$, which amounts to counting the complementary minors in the sum that have nonzero $x_u$ term.
\subsection{Classifying 321-, 1324-avoiding, 2143-containing permutations}
We will show such permutations satisfy one of two prescribed block structures.
\begin{prop}\label{2143Patterns}
Let $w\in \mathfrak{S}_n$ be a permutation that avoids $321$, $1324$, and contains $2143$. 

Let $a' = w^{-1}(1) - 1, b' = w(1) - 1, c' = n - w(n), d' = n - w^{-1}(n)$. We reserve letters $a, b, c, d$ for the size of finer blocks. We will see that in one case, we will have blocks whose sizes are $a', b', c', d'$ respectively. In the other case, we will break up $a', b', c', d'$ into six values which will form the lengths of our blocks.

\begin{enumerate}
    \item If $a' + b' + c' + d' \leq n,$ let $a = a', b = b', c = c', d = d',$ and let $e=n-a-b-c-d \ge 0$. Then the one-line notation of $w$ is
    \begin{equation*}
        (b+1..b+a : 1..b : b+a+1..b+a+e : n-c+1..n : n-c-d+1..n-c).
    \end{equation*}
In this case, we see $w$ has block structure $[2][1][3][5][4]$ with block lengths $a, b, e, c, d$ and has fixed points at every $i \in [b+a+1, b+a+e]$. Only the $[3]$ block is allowed to be empty.

\item Otherwise, there exist integers $1 \le a, b, c, d \le n$, $0 \le e, f \le n$ with $\max(e, f) \ge 1$, such that $a + e = a', b + f = b', c + e = c'$, $d + f = d'$, $a+b+c+d+e+f=n$, and the one-line notation of $w$ is
\begin{equation*}
    (b+f+1..b+f+a : n-c-e+1..n-c : 1..b : n-c+1..n : b+1..b+f : n-d-c-e+1..n-c-e).
\end{equation*}
In this case, $w$ has block structure $[3][5][1][6][2][4]$ with block lengths $a, e, b, c, f, d$, and the middle two blocks do not contain any numbers in $[b'+1, n-c']$. Only one of the $[5]$ or $[2]$ blocks are allowed to be empty.
\end{enumerate}
\end{prop}

In fact, if $w$ is in the second case, then it must contain one of two prescribed patterns.
\begin{lemma}\label{lem:24153, 31524 pattern}
Suppose that $w$ is a permutation avoiding $1324, 321$ that contains the pattern $2143,$ and define $a', b', c', d'$ as in Proposition \ref{2143Patterns}. Suppose that $a' + b' + c' + d' > n.$ Then, $w$ either contains the pattern $24153$ or $31524.$
\end{lemma}

\begin{proof}[Proof of Proposition \ref{2143Patterns}]

\par We first extract information from the pattern avoidance conditions on $w$. 
\begin{itemize}
    \item \textbf{Observation 1.} Since $w$ avoids $321, 1324$ and contains $2143$, by Lemma \ref{w(1) and w(n) inequality for 2143}(a,c) we have that $a', b', c', d' \neq 0$, $n - d' > a' + 1$ and $n - c' > b' + 1$.
    
    \item \textbf{Observation 2.} Since $w$ avoids $321$, we have $w([1, a']) \subset [b'+1, n]$ and $w([n-d', n]) \subset [1, n-c']$. For instance, to prove the first claim, notice that if $1 \leq y \leq a'$ satisfies $w(y) \leq b',$ then $1, y, a' + 1$ forms a $321$-pattern. (See also Proposition \ref{321-avoiding Rectangles}.)

    \item \textbf{Observation 3.} Since $w$ avoids $1324$, we have that the following six sequences are increasing:
    \begin{equation*}
        w(1..a'), w(a'+1..n-d'), w(n-d'+1..n), w^{-1}(1..b'), w^{-1} (b'+1..n-c'), w^{-1} (n-c'+1..n).
    \end{equation*}
    For example, if the first sequence $(w(1), w(2), \cdots, w(a'))$ is non-increasing, say we have $1 \le x < y \le a'$ such that $w(x) > w(y)$, then Observation 2 tells us that $x > 1$, and then $1, x, y, n-d'$ forms a $1324$-pattern.

    \item \textbf{Observation 4.} Since $w$ avoids $1324$, there cannot exist $y \in [1, a' : n-d'+1, n]$ and $x \in [a'+1, n-d']$ such that $w(y) \notin [b'+1, n-c']$ and $w(x) \in [b'+1, n-c']$. For if say $y \in [1, a']$, then $w(y) \ge b' + 1$ by Observation 2, and so $w(y) \ge n-c' + 1$. Then $1, y, x, n-d'$ form a $1324$-pattern. A similar contradiction holds if $y \in [n-d'+1, n]$.
\end{itemize}

\par Now, we begin the proof of the proposition in earnest. Let $a \le a'$ be the largest value such that $w(a) \le n-c'$. Then $[1, a] \subset w^{-1} ([b'+1, n-c'])$ and $w^{-1} (b'+1..n-c')$ is increasing by Observation 3, so $(1..a) = w^{-1} (b'+1..b'+a)$ and $w(1..a) = (b'+1..b'+a)$. Now by maximality we have $w(x) > n-c'$ for $a < x \le a'$, so $[a+1, a'] \subset w^{-1} ([n-c'+1, n])$. Since $w^{-1}(n-c'+1..n)$ is increasing, we have $(a+1..a') = w^{-1} (n-c'+1..n-c'+a'-a)$. Thus, $w(a+1..a') = (n-c'+1..n-c'+a'-a)$. In conclusion, so far we have determined the values $w(1)$ through $w(a')$.
\par We can also similarly define $d \le d'$ to be the largest value such that $w(n-d) \ge b'+1$. Then an analogous argument shows that $w(n-d+1..n) = (n-d-c'+1..n-c')$ and $w(n-d'+1..n-d) = (b'-(d'-d)+1..b')$.
\par Since $w([1, a]) = [b'+1,b'+a]$ and $w([n-d+1,n]) = [n-d-c'+1,n-c']$ are disjoint and $b'+1 < n-c'$ by Observation 1, we must have $b'+a \le n-d-c'$.
\par Finally, it suffices to determine the sequence $w(a'+1..n-d')$. By Observation 3, the sequence is increasing and consists of the elements in the set $S = [n] \setminus w([1,a';n-d'+1,n]).$ It will simplify matters to define $b := b' - (d'-d)$ and $c := c' - (a'-a)$. From our computation of $w(1..a':n-d'+1..n)$, we see that $S = [1, b : b'+a+1, n-d-c' : n - c + 1 : n]$. Thus, we can compute $w(a'+1..a' + b) = (1.. b),$
$w(a' + b+1..n-d'-c) = (b'+a+1..n-d-c')$, and
$w(n-d'-c+1..n-d') = (n-c+1..n)$. Since $w(a'+1) = 1$, we must have $b > 0$; similarly, $c > 0$.
We divide into two cases.
\par \textbf{Case 1.} Suppose $a = a'$ and $d = d'$; then $b = b'$, $c = c'$. If we define $e = n - a - b - c - d$, then we get our desired result. In particular, $a' + b' + c' + d' = n-e \le n$.
\par \textbf{Case 2.} Suppose $a < a'$ or $d < d'$. Then at least one of $w(a')$ or $w(d')$ does not lie in $[b'+1, n-c']$, so by Observation 4, we have $w([a'+1, n-d']) \cap [b'+1, n-c'] = \emptyset$. But we know that $w([a'+b+1, n-d-c']) = [b'+a+1,n-d-c']$, which will be a contradiction unless both sides are empty. Thus, $a+b' \ge n-d-c'$, and so $a+b' = n-d-c'$. If $e := a' - a$, $f := d' - d$, then we get $b := b' - f$, and $c := c' - e$, and we recover our desired result. Finally, we note that $a' + b' + c' + d' > a + b' + c' + d = n$.

Thus, we have showed Case 1 happens if and only if $a' + b' + c' + d' \le n$. This completes the proof of the proposition.

\end{proof}

\subsection{Temperley-Lieb Immanants as a Sum of Complementary Minors}\label{CompliMinors}
\par Let $w \in \fkS_n$ be a permutation that avoids $321$ and $1324$ but contains $2143$. Our main goal in this subsection is to express $\imm_w$ as an explicit sum of certain complementary minors. To give this explicit expression of complementary minors, we split up our argument into two cases, along the lines of the cases provided in Proposition \ref{2143Patterns}. For each of these cases, we closely follow the method in Proposition \ref{Rectangle Unique Matching}. We first prove that a unique non-crossing matching exists for a given condition of colorings (Lemmas \ref{non-crossing matching of case1} and \ref{non-crossing matching of case2}), and then we show that the unique non-crossing matching corresponds to $w$ (Lemmas \ref{ncm equals w case1} and \ref{ncm equals w case2}). Thus, we are able to construct an explicit set $\cS$ of colorings such that for each $u \neq w$, there exists a ``sign-reversing'' involution $\iota_u$ on $\cS$. Finally, we will construct a special linear combination of complementary minors corresponding to colorings in $\cS$ that equals our desired $\imm_w$. To prove the equality, we use $\iota_u$ to pair up opposite terms to generate cancellation (Propositions \ref{sumcm5} and \ref{sumcm6}).

First, we present the following general lemma concerning non-crossing matchings:
\begin{lemma}\label{general non-crossing matching}
Let $a, b, c, d, e \ge 0$ be integers with $n = a + b + c + d + e$. Consider vertices labelled $1, 2, \ldots, 2n.$ Then there exists a unique coloring of these vertices using the colors white and black and a non-crossing matching compatible with this coloring such that the following conditions hold. 
\begin{enumerate}
    \item $[1, b+c+e]$ are colored black,
    
    \item in the interval $[b+c+e+1, a+2b+c+e]$, there are exactly $a$ black and $b$ white vertices, and there are no pairings between two vertices in this interval (which we will refer to as an ``internal pairing''),
    
    \item $[a+2b+c+e+1, a+b+e+n]$ are colored white,
    
    \item in the interval $[a+b+e+n+1, 2n]$, there are exactly $d$ black and $c$ white vertices, and there are no internal pairings in this interval.
\end{enumerate}
\end{lemma}
As an example, this is the case where $a = 2$ and $b = c = d = e = 1,$ with the red lines dividing the four different ranges of indices with the different conditions. 
\begin{center}
    \includegraphics[scale=0.7]{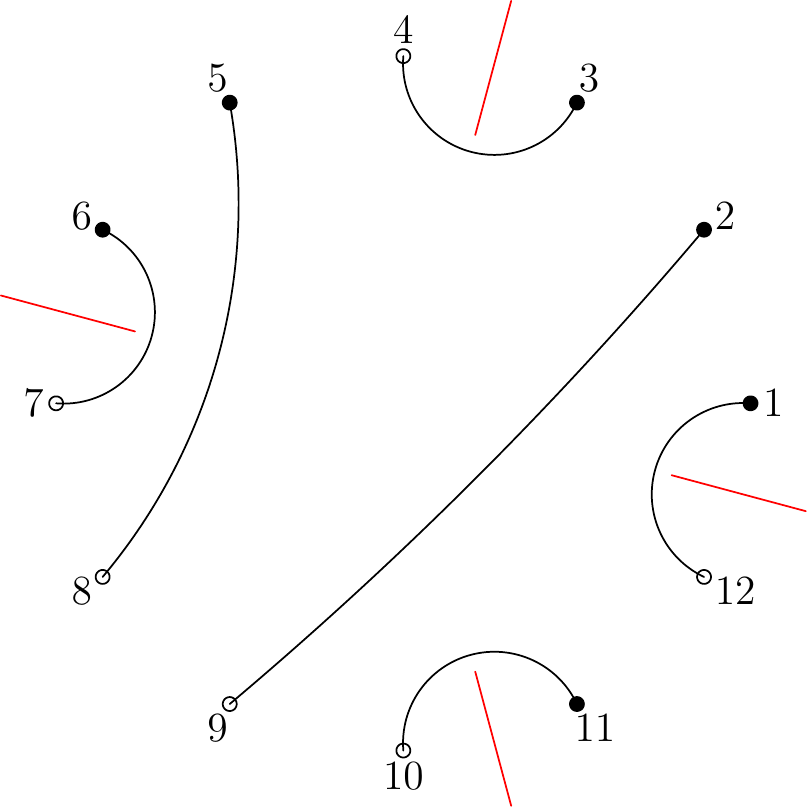}
\end{center}
\begin{remark}
We represent this lemma by drawing these vertices on a circle to illustrate the symmetry. When we then apply this lemma, we will pick two ways of choosing a diameter and letting the vertices on each side of the diameter form a column of $n$ vertices, consistent with the pictures we have previously shown to represent the non-crossing matchings.
\end{remark}

\par We specialize Lemma \ref{general non-crossing matching} to our situation. The next two lemmas may seem mysterious at first sight; a curious reader can turn to the discussion before Lemma \ref{ncm equals w case1} for their significance.

\begin{lemma}\label{non-crossing matching of case1}
Let $a, b, c, d, e$ be nonnegative integers, so $a, b, c, d \geq 1$ and $a+b+c+d+e=n$. Then, there is a unique non-crossing matching and a coloring compatible with it, such that the coloring satisfies
\begin{enumerate}
    \item $i$ is black for $i\in [a+1,n-d]$
    \item $i'$ is white for $i\in [b+1, n-c]$
    \item There are exactly $a$ black vertices and $b$ white vertices in $[1,a]\cup [1,b]'$
    \item There are exactly $d$ black vertices and $c$ white vertices in $[n-d+1,n]\cup [n-c+1,n]'$ 
    \item There are no pairings between two vertices in $[1,a]\cup [1,b]'$ (which we will refer to as an ``internal pairing")
    \item There are no internal pairings in $[n-d+1,n]\cup [n-c+1,n]'$
\end{enumerate}
\end{lemma}
For instance, here is a complementary coloring, with the corresponding unique matching, with $a = c = 2, e = b = d = 1.$ The boxed vertices are those that are fixed by conditions \textit{1} and \textit{2}.
\begin{center}
    \includegraphics[]{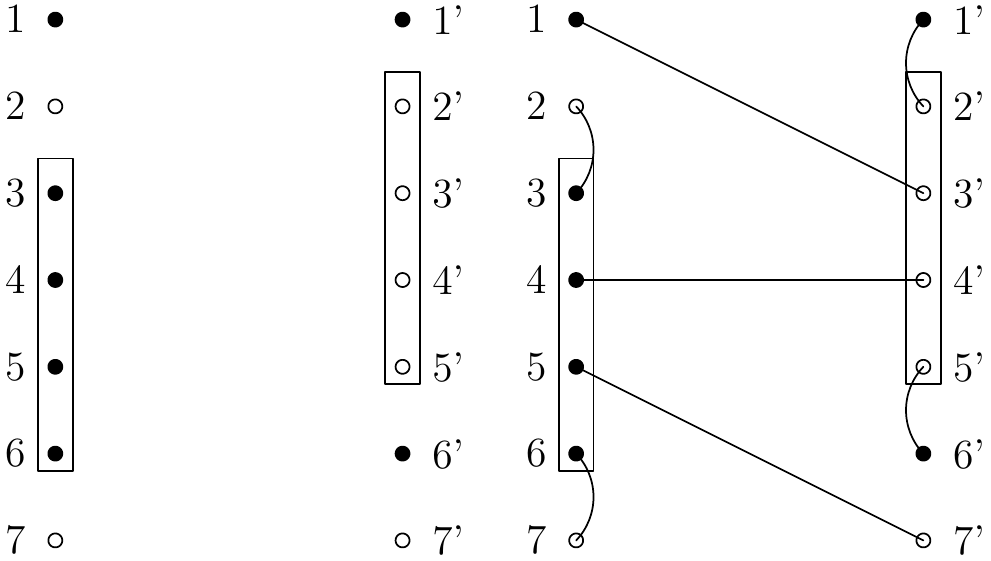}
\end{center}

\begin{lemma}\label{non-crossing matching of case2}
Let $a, b, c, d, e, f$ be nonnegative integers where $a, b, c, d, \max(e, f) \geq 1,$ and $a+b+c+d+e+f=n$. Then, there is a unique non-crossing matching and a coloring compatible with it, such that the coloring satisfies
\begin{enumerate}
    \item $i$ is black for $i\in [1,a+e]$
    \item $i$ is white for $i\in [a+e+b + c+1, n]$
    \item $i'$ is black for $i\in [1,b+f]$
    \item $i'$ is white for $i\in [b+f+a+d+1,n]$
    \item There are exactly $c$ black vertices and $b$ white vertices in $[a+e+1,a+e+b+c]$
    \item There are exactly $d$ black vertices and $a$ white vertices in $[b+f+1,b+f+a+d]'$ 
    \item There are no pairings between two vertices in $[a+e+1,a+e+b+c]$ (which we will refer to as an ``internal pairing")
    \item There are no internal pairings in $[b+f+1,b+f+a+d]'$
\end{enumerate}
\end{lemma}
This case has the following diagram, where $a = c = 2, b = d = e = f = 1,$ with the boxed vertices again being those fixed (this time for conditions \textit{1} to \textit{4}).
\begin{center}
    \includegraphics[]{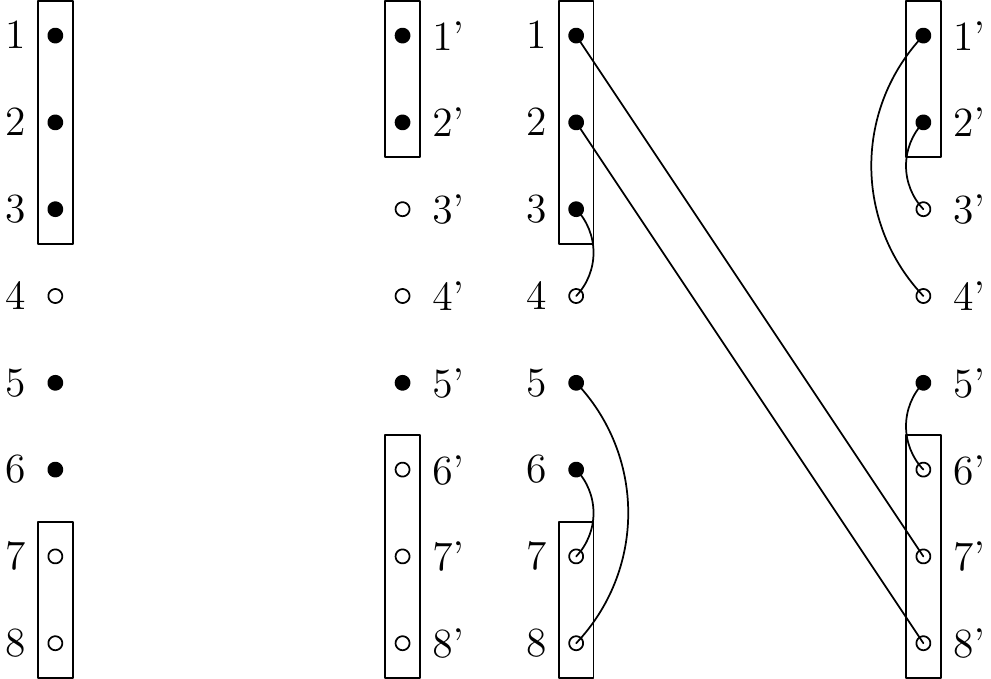}
\end{center}

\par As promised, we reveal that the non-crossing matching in Lemmas \ref{non-crossing matching of case1} and \ref{non-crossing matching of case2} actually belong to one of the $w$'s described in Proposition \ref{2143Patterns}.

\begin{lemma}\label{ncm equals w case1}
Let $w \in \fkS_n$ have block structure $[2][1][3][5][4]$ with block lengths $a,b,e,c,d$, as stated in the first case of Proposition \ref{2143Patterns}. The non-crossing matching of $w$ is exactly the non-crossing matching in Lemma \ref{non-crossing matching of case1}.
\end{lemma}

\begin{lemma}\label{ncm equals w case2}
Let $w \in \fkS_n$ have block structure $[3][5][1][6][2][4]$ with block lengths $a,e,b,c,f,d$, as stated in the second case of Proposition \ref{2143Patterns}. The non-crossing matching of $w$ is exactly the non-crossing matching in Lemma \ref{non-crossing matching of case2}.
\end{lemma}

Using Lemmas \ref{non-crossing matching of case1} through \ref{ncm equals w case2}, we are now able to describe the Temperley-Lieb immanant of $w$ in terms of complementary minors. Again, we present the two cases separately.

\begin{prop}\label{sumcm5}
Let $w \in \fkS_n$ have block structure $[2][1][3][5][4]$ with block lengths $a,b,e,c,d$, as stated in the first case of Proposition \ref{2143Patterns}. Define $\cS$ to be the set of all possible colorings $(I, J)$ satisfying conditions \textit{1} through \textit{4} in Lemma \ref{non-crossing matching of case1}.
Then, we have
\begin{equation}\label{antidiagcase1}
    \imm_w = \sgn(w) (-1)^n \sum\limits_{(I,J) \in \cS} (-1)^{|I|} \cm_{I,J},
\end{equation} 
\end{prop}

\begin{remark}
Explicitly, we have
\begin{equation}\label{antidiagcase1'}
    \imm_w = \sgn(w) \sum\limits_{I_1, I_2, I_3, I_4} (-1)^{|I_1| + |I_3|} \cm_{I_1 \cup I_3, I_2 \cup I_4}.
\end{equation}
where the sum runs over all $I_1 \subset [1, a], I_2 \subset [1, b], I_3 \subset [n-d +1, n], I_4 \subset [n-c+1, n]$ satisfying $|I_1| = |I_2|, |I_3| = |I_4|$.

For instance, with $w = 2143,$ we are given that $a = b = c = d = 1,$ and so the possible $(I_1 \cup I_3, I_2 \cup I_4)$ are:
\begin{equation*}
    (\{1,4\}, \{1',4'\}), (\{1\}, \{1'\}), (\{4\}, \{4'\}), (\emptyset, \emptyset).
\end{equation*}
(It is a coincidence in this example that we always have $I_1 = I_2$ and $I_3 = I_4$.) They correspond to the following complementary minors:
\begin{center}
    \includegraphics[]{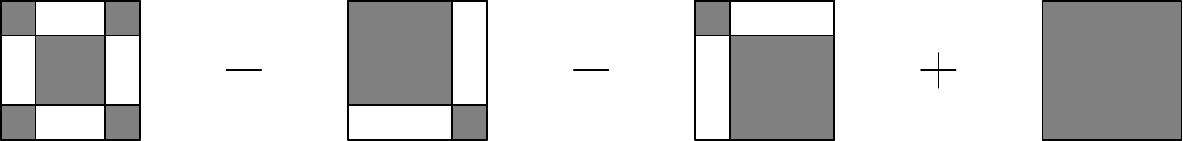}
\end{center}
\end{remark}
Notice that $I_1, I_2$ could both be taken to be empty.
\begin{proof}
Let $\beta(I, J) := s(I) + s(J) + |I|$, where $s(I) = \sum_{i \in I} i$. First, we claim for some $\alpha$,
\begin{equation}\label{antidiagcase1eqn}
    \imm_w = \alpha \sum\limits_{(I,J) \in \cS} (-1)^{\beta(I,J)} \Delta_{\overline{I}, \overline{J}} \Delta_{ I, J },
\end{equation}
Here, $\cS$ is the set of all possible colorings $(I, J)$ satisfying conditions \textit{1} through \textit{4} in Lemma \ref{non-crossing matching of case1}. (Recall that in a coloring $(I, J)$, we have that $I, \overline{J}'$ are colored black and $\overline{I}, J'$ are colored white.) 
\par To show \eqref{antidiagcase1eqn}, we expand each product of complementary minors $(-1)^{\beta(I, J)} \Delta_{I,J} \Delta_{\bI,\bJ}$ into a sum of $\imm_u$ terms, where $\imm_u$ appears with coefficient $(-1)^{\beta(I, J)}$ iff $(I, J) \in \cS$ by Proposition \ref{prop:cm equals imm sum}. Thus, the RHS of \eqref{antidiagcase1eqn} can be expressed as $\sum_u c_u \imm_u$, where the sum ranges over $321$-avoiding permutations in $\fkS_n$, and the coefficient $c_u = \sum_{(I, J) \in \cS_u} (-1)^{\beta(I,J)}$, where $\cS_u$ is the set of colorings in $\cS$ compatible with $u$. We want to show $c_u = 0$ unless $u = w$.

\par Suppose there exists $u\not = w$ such that $c_u \neq 0$. Then $\cS_u$ is non-empty, so by Lemmas \ref{non-crossing matching of case1} and \ref{ncm equals w case1}, the non-crossing matching of $u$ does not satisfy both conditions \textit{5} and \textit{6} in Lemma \ref{non-crossing matching of case1}. Fix a pair of vertices $v_1, v_2$ that form an internal pairing (so both vertices are in $[1, a] \cup [1, b]',$ or both are in $[n - d + 1, n] \cup [n - c + 1, n]'$). Now, we define the involution $\iota$ on colorings that swaps the colors of $v_1, v_2$. We will now show that $\iota$ maps $\cS_u$ to $\cS_u$ and $\beta(\iota(I, J)) \equiv \beta(I, J) + 1 \pmod 2$. Indeed, conditions \textit{1} and \textit{2} are preserved under $\iota$ since $v_1, v_2 \notin [a+1, n-d] \cup [b+1, n-c]'$, and conditions \textit{3} and \textit{4} are preserved because $v_1, v_2$ form an internal pairing in either $[1, a] \cup [1, b]'$ or $[n-d+1, n] \cup [n-c+1, n]'$. Furthermore, $v_1$ and $v_2$ are paired in $\ncm(u)$, so swapping the colors of $v_1$ and $v_2$ will preserve compatibility of the coloring with $u$. This shows $\iota$ maps $\cS_u$ to $\cS_u$. For the other claim, we have two main cases to consider.
\par First, if $v_1, v_2$ are both unprimed, then from our non-crossing matching they must be opposite colors; say $v_1 \in I$ and $v_2 \in \overline{I}$. Also, since $v_1, v_2$ are paired, their labels have different parities (when viewed as integers from $1$ to $n$). Now, swapping the colors of $v_1, v_2$ doesn't change the size of $I$ (the number of unprimed black vertices) but replaces $v_1$ with $v_2,$ so this swap changes $\beta(I, J)$ by $v_2 - v_1$. But $v_2 - v_1$ is odd by Remark 2 after Definition \ref{def:ncm}, so $\beta(\iota(I, J)) \equiv \beta(I, J) + 1 \pmod 2$. The same argument holds if both are primed, but now considering $J$.
\par Otherwise, suppose that $v_1$ is primed and $v_2$ is unprimed. Then, $v_1, v_2$ must have the same parity. In this case, swapping the colors of $v_1, v_2$ will add $v_1$ to $I$ and $v_2$ to $J$, or remove each of them from their respective sets; so from the parities of $v_1, v_2,$ the sum of the elements will be the same parity. However, the size $|I|$ also changes by $1$, and so $\beta(I, J)$ changes by an odd number.
\par In either case, we have $\beta(\iota(I, J)) \equiv \beta(I, J) + 1 \pmod 2$. Thus, since $\iota : \cS_u \to \cS_u$ is an involution,
\begin{equation}\label{eqn:cancellation}
    c_u = \sum_{(I, J) \in \cS_u} (-1)^{\beta(I, J)} = \sum_{(I, J) \in \cS_u} (-1)^{\beta(\iota(I, J))} = \sum_{\iota(I, J) \in \cS_u} (-1)^{\beta(I, J)+1} = -c_u.
\end{equation}
This shows $c_u = 0$, contradicting our original assumption that $c_u \neq 0$. Thus, in fact $c_u = 0$ for $u \neq w$, giving equation \eqref{antidiagcase1eqn} up to a global factor.

Next, conditions \textit{1} through \textit{4} in Lemma \ref{non-crossing matching of case1} tell us that a coloring $(I, J) \in \cS$ can be expressed in the form $\overline{I} = I_1 \cup I_3$ and $\overline{J} = I_2 \cup I_4$, where $I_1 \subset [1, a]$, $I_3 \subset [n-d+1, n]$, $I_2 \subset [1, b]$, $I_4 \subset [n-c+1, n]$, and $|I_1| = |I_2|, |I_3| = |I_4|$. By Lemma \ref{cmminors}, we can rewrite \eqref{antidiagcase1eqn} as
\begin{equation}\label{antidiagcase1eqn2}
    \imm_w = \alpha \sum\limits_{(I,J) \in \cS} (-1)^{|I|} \cm_{I,J} = \alpha \sum\limits_{I_1, I_2, I_3, I_4} (-1)^{n-|I_1| - |I_3|} \cm_{I_1 \cup I_3, I_2 \cup I_4}.
\end{equation}
To determine $\alpha$, we compare coefficients of $x_w$ in $\eqref{antidiagcase1eqn2}$. By the explicit formula of Definition \ref{def:cm}, we know that $\cm_{I_1 \cup I_3, I_2 \cup I_4}$ has a nonzero $x_w$ coefficient iff $w(I_1 \cup I_3) = I_2 \cup I_4$. This forces $I_1 = I_2 = I_3 = I_4 = \emptyset,$ since from Proposition \ref{2143Patterns} we know that $w(I_1) \subset w([1, a]) = [b+1, a+b]$ and $w(I_3) \subset w([n-d+1, n]) = [n-d-c+1, n-c],$ both of which are disjoint from $[1, b] \cup [n-c+1, n] \supset I_2 \cup I_4$. Then since $f_w (w) = 1$, we get $1 = \alpha \sgn(w) (-1)^n$, so $\alpha = \sgn(w) (-1)^n$. With this $\alpha$, \eqref{antidiagcase1eqn2} is equivalent to \eqref{antidiagcase1} and \eqref{antidiagcase1'}, as desired.
\end{proof}


\begin{prop}\label{sumcm6}
Let $w \in \fkS_n$ have block structure $[3][5][1][6][2][4]$ with block lengths $a,e,b,c,f,d$, as stated in the second case of Proposition \ref{2143Patterns}. Let $\cS$ be the set of possible colorings satisfying conditions \textit{1} through \textit{6} in Lemma \ref{non-crossing matching of case2}. Then, we have
\begin{equation}\label{antidiagcase2}
    \imm_w = \sgn(w) \sum\limits_{(I,J) \in \cS} \cm_{I,J}.
\end{equation}
\end{prop}

\begin{remark}
Explicitly, we have
\begin{equation}\label{antidiagcase2'}
    \imm_w = \sgn(w) \sum\limits_{I_1,I_2} \cm_{[1,a+e] \cup I_1 , I_2 \cup [b+f+a+d+1,n] },
\end{equation}
where the sum runs over all $I_1\subset [a+e+1,a+e+b+c]$ with $|I_1|=c$ and $I_2\subset [b+f+1,b+f+a+d]$ with $|I_2|=a$. 

For instance, with $w = 24153,$ we are given that $a = b = c = d = f = 1$ and $e = 0,$ and so the possible $([1, a+e] \cup I_1, I_2 \cup [b + f + a + d + 1, n])$ are:
\begin{equation*}
    (\{1,2\}, \{3,5\}), (\{1, 2\}, \{4, 5\}), (\{1, 3\}, \{3, 5\}), (\{1, 3\}, \{4, 5\}).
\end{equation*}
They correspond to the following complementary minors:
\begin{center}
    \includegraphics[]{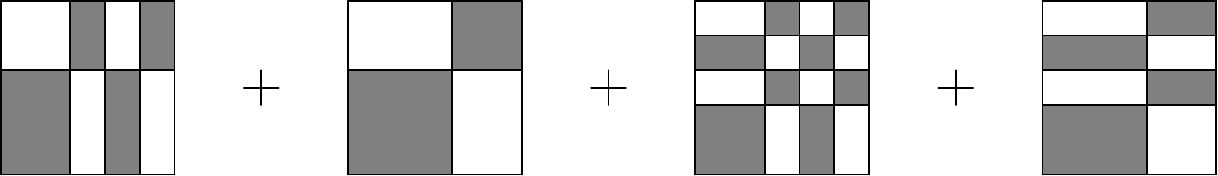}
\end{center}
\end{remark}
\begin{proof}
Let $\beta(I, J) := s(I) + s(J)$. First, we claim for some $\alpha$,
\begin{equation}\label{antidiagcase2eqn}
    \imm_w = \alpha \sum\limits_{(I,J) \in \cS} (-1)^{\beta(I,J)} \Delta_{I, J} \Delta_{\overline{I}, \overline{J}}.
\end{equation}
Here, $\cS$ is the set of possible colorings satisfying conditions \textit{1} through \textit{6} in Lemma \ref{non-crossing matching of case2}.
The proof of \eqref{antidiagcase2eqn} is extremely similar to the proof of \eqref{antidiagcase1eqn}, so we sketch the details. First, we expand each product of complementary minors $(-1)^{\beta(I, J)} \Delta_{I,J} \Delta_{\bI,\bJ}$ into a sum of $\imm_u$ terms, where $\imm_u$ appears with coefficient $(-1)^{\beta(I, J)}$ iff $(I, J) \in \cS$ by Proposition \ref{prop:cm equals imm sum}. Thus, the RHS of \eqref{antidiagcase2eqn} can be expressed as $\sum_u c_u \imm_u$, where the sum ranges over $321$-avoiding permutations in $\fkS_n$, and the coefficient $c_u = \sum_{(I, J) \in \cS_u} (-1)^{\beta(I,J)}$, where $\cS_u$ is the set of colorings in $\cS$ compatible with $u$. We want to show $c_u = 0$ unless $u = w$. Suppose that $u \neq w$ and $c_u \neq 0$; then $\cS_u \neq \emptyset$. By Lemmas \ref{non-crossing matching of case2} and \ref{ncm equals w case2}, the non-crossing matching of $u$ does not satisfy both conditions \textit{7} and \textit{8} in Lemma \ref{non-crossing matching of case2}, so there is some internal pairing of vertices $v_1, v_2;$ note that both are either primed or unprimed.
\par Then, we know that $v_1, v_2$ must have different parities, so swapping the colors of $v_1, v_2$ will change the parity of $\beta(I,J)$. Also, since $v_1, v_2$ form an internal pairing, conditions \textit{1} through \textit{6} of \ref{non-crossing matching of case2} and compatibility with $u$ still hold after swapping the colors of $v_1, v_2$. In summary, we have constructed a sign-reversing involution $\iota : \cS_u \to \cS_u$. Using \eqref{eqn:cancellation}, we see that $c_u = 0$, a contradiction. Thus in fact $c_u = 0$ whenever $u \neq w$, so we obtain equation \eqref{antidiagcase2eqn} up to a global sign.

\par Next, conditions \textit{1} through \textit{6} in Lemma \ref{non-crossing matching of case2} tell us that a coloring $(I, J) \in \cS$ can be expressed in the form $I = [1,a+e] \cup I_1$ and $J = I_2 \cup [b+f+a+d+1,n])$ where $|I_1| = c$, $I_1 \subset [a+e+1, a+e+b+c]$, $|I_2| = a$, $I_2 \subset [b+f+1, b+f+a+d]$. By Lemma \ref{cmminors}, we can rewrite \eqref{antidiagcase2eqn} as
\begin{equation}\label{antidiagcase2eqn2}
    \imm_w = \alpha \sum\limits_{(I,J) \in \cS} \cm_{I,J} = \alpha \sum\limits_{I_1,I_2} \cm_{[1,a+e] \cup I_1 , I_2 \cup [b+f+a+d+1,n] }
\end{equation}
To determine $\alpha$, we compare coefficients of $x_w$ in $\eqref{antidiagcase2eqn2}.$ Note that $f_w (w) = 1,$ and by the explicit formula of Definition \ref{def:cm}, any complementary minor $\cm_{[1,a+e] \cup I_1 , I_2 \cup [b+f+a+d+1,n] }$ with a nonzero $x_w$ term requires $I_1$ to contain $w^{-1}([b + f + a + d + e + 1, n]) = [a + e + b + 1, a + e + b + c]$ and $I_2$ to contain $w([1, a]) = [b + f + 1, b + f + a]$ (from Proposition \ref{2143Patterns}). But since $|I_1| = c$ and $|I_2| = a$, we must have $I_1 = [a + e + b + 1, a + e + b + c]$ and $I_2 = [b + f + 1, b + f + a]$. Thus, there is a unique choice for $I_1, I_2$ to get a nonzero $x_w$ term, which means $1 = \alpha \sgn(w)$, so $\alpha = \sgn(w)$. With this $\alpha$, \eqref{antidiagcase2eqn2} is equivalent to \eqref{antidiagcase2} and \eqref{antidiagcase2'}, as desired.


\end{proof}

Using the above propositions, we can explicitly express $\imm_w$ as a sum of products of complementary minors. From there, it becomes easy to calculate $f_w(u)$ for each $u\in \mathfrak{S}_n.$


\begin{thm}\label{AntiDiagCoeff} Let $w \in \fkS_n$ have block structure $[2][1][3][5][4]$ with block lengths $a,b,e,c,d$, as stated in the first case of Proposition \ref{2143Patterns}. Let $u \in \mathfrak{S}_n$. Then the coefficient of $x_u$ in $\imm_w$ is given by
\[
f_w(u) = 
\begin{cases}
0 & \text{if there exists } i \in [1,a] \text{ s.t. } u(i) \in [1,b],\\
& \text{or there exists } i \in [n+1-d, n]\text{ s.t. }  u(i) \in [n+1 - c, n] \\
&\\
\sgn(w)\sgn(u)\binom{A+B}{A} & \text{otherwise, where A = } |[1, a] \cap u^{-1}([n+1-c, n])|, \\
& B = |[1, b] \cap u^{-1}([n+1-d, n])|.
\end{cases}
\]

\end{thm}

\begin{proof}
Let $\cS$ be the set of colorings that satisfy conditions \textit{1} through \textit{4} in Lemma \ref{non-crossing matching of case1}, and let $\cS'_u$ be the set of colorings in $\cS$ such that $i$ and $u(i)'$ have different colors for all $i\in [n]$. (This is different from the $\cS_u$ defined in the proof of Proposition \ref{sumcm5}.)

Consider a $\cm_{I,J}$ term in the sum given in Proposition \ref{sumcm5} (so the coloring $(I, J) \in \cS$). By the remark after Definition \ref{def:cm}, notice that $x_u$ has nonzero coefficient in $\cm_{I,J}$ if and only if the coloring $(I, J) \in \cS_u'$ (in which case the coefficient is $\sgn(u)$). Thus, we get
\begin{equation}\label{eqn:extract_coeff1}
    f_w (u) = (-1)^n\sgn(w) \sgn(u) \sum_{(I, J) \in \cS_u'} (-1)^{|I|}.
\end{equation}

First, if there exists $i \in [1,a]$ such that $u(i) \in [1,b]$, then swapping the colors of $i$ and $u(i)'$ is an involution $\cS_u' \to \cS_u'$ that flips the sign of $(-1)^{|I|}$. By a cancellation argument similar to \eqref{eqn:cancellation}, we get $f_w (u) = 0$. (Alternate algebraic approach: we have $|u([1, a]) \cap [1, b]| = 1 > 0 = |w([1, a]) \cap [1, b]|$ by assumption and Proposition \ref{2143Patterns} case 1, so $u \not\ge w$ by equivalent definition \ref{bruhat3} of Bruhat order. Thus, $f_w (u) = 0$ by Lemma \ref{lem:basic fwu}.) A similar argument holds if there exists $i \in [n+1-d,n]$ such that $u(i) \in [n+1-c, n]$.

Thus, assume neither condition holds. We will show two claims: (1) $|\cS'_u| = \binom{A+B}{A}$ and (2) $(-1)^{n-|I|} = 1$ for each $(I, J) \in \cS'_u$.

Consider the $n$ pairs of vertices $(i, u(i)');$ each pair must consist of one white and one black vertex to have a nonzero $x_u$ term. Then $a-A$ of the vertices in $[1,a]$ are paired with a vertex in $[b+1, n-c]'$ and thus are black. Similarly, $b-B$ of the vertices in $[1,b]'$ are paired with a vertex in $[a+1, n-d]$ and are therefore white. As a result, among the $A+B$ remaining unresolved vertices in $[1,a] \cup [1,b]'$, $A$ must be black and $B$ must be white. There are $\binom{A+B}{A}$ ways of choosing colors for the unresolved vertices. 

Note that each coloring of the unresolved vertices in $[1, a] \cup [1, b]'$ uniquely specifies the entire coloring: since $u^{-1}([1, b]) = [a + 1, n]$ by assumption, we have that the coloring of the unresolved vertices in $[1, b]'$ determines the color of $b - B$ vertices in $[n - d + 1, n].$ But the remaining elements in $[n - d + 1],$ again by assumption, must be sent by $u$ to some element in $[b+1, n-c],$ and so their color is uniquely determined; the same argument holds for unresolved vertices in $[1, a]$ and those in $[n-c+1, n]'.$ Each of these colorings, by construction, colors $i, u(i)'$ different and satisfies conditions \textit{1} through \textit{3} in Lemma \ref{non-crossing matching of case1}. Condition \textit{4} follows from the first three: with $B$ vertices in $[1, b]'$ paired with those in $[n - d + 1, n],$ the remaining $d - B$ are necessarily colored black, and similarly $c - A$ vertices in $[n - c + 1, n]'$ are necessarily colored white. Of the remaining $A + B$ vertices, $A$ are white and $B$ are black (by the reverse condition imposed in $[1, a] \cup [1, b]'$), yielding $c$ white and $d$ black vertices. In summary, each of our choices for the unresolved vertices leads to a valid coloring in $\cS'_u$, and so $|\cS'_u| = \binom{A+B}{A}$, proving claim 1.

Finally, suppose that $C$ of the unresolved vertices in $[1, a]$ are white. Then, $A - C$ of the unresolved vertices in $[1, a]$ are black, since there are $A$ unresolved vertices in $[1, a].$ Thus, there are $A - (A - C)$ many black vertices in $[1,b]',$ since there are $A$ unresolved vertices in $[1, a] \cup [1,b]'$ that are colored black. But since $u^{-1}([1,b])$ lies in $[a+1, n],$ $i'$ being unresolved in $[1, b]'$ means that $u^{-1}(i)$ is unresolved in $[n - d + 1, n],$ and furthermore this is a $1-1$ correspondence. We thus observe that $C$ of the unresolved vertices in $[n-d+1, n]$ are white. Thus, as the only white vertices in $[1, a] \cup [n-d+1, n]$ are unresolved by our above argument, we have that $(-1)^{n-|I|} = (-1)^{2C} = 1,$ since the vertices in $I$ are colored black. This proves Claim 2.

Plugging in Claims 1 and 2 into \eqref{eqn:extract_coeff1}, the coefficient $f_w(u)$ of $x_u$ in $\imm_w$ is equal to $\binom{A + B}{A} \sgn(w)\sgn(u),$ as desired.
\end{proof}

\begin{thm}\label{coefficient of u in w case 2}
Let $w \in \fkS_n$ have block structure $[3][5][1][6][2][4]$ with block lengths $a,e,b,c,f,d$, as stated in the second case of Proposition \ref{2143Patterns}. Let $u\in \mathfrak{S}_n$. Then
\[f_w(u) = \begin{cases}
0 & \text{if there exists } i \in [1,a+e] \text{ s.t. } u(i) \in [1,b+f],\\
& \text{or there exists }i \in [a+e+b+c+1, n]\text{ s.t. }  u(i) \in [b+f+a+d+1, n] \\
&\\
\sgn(w)\sgn(u)\binom{A+B}{A} & \text{otherwise, where} \\
& A=c- |[a+e+1,a+e+b+c] \cap u^{-1}([b+f+a+d+1,n])|,\\
& B= b- |[a+e+1,a+e+b+c] \cap u^{-1}([1,b+f])|.
\end{cases}
\]
\end{thm}

\begin{remark}
The binomial coefficient $\binom{A+B}{A}$ is taken to be zero if $A < 0$ or $B < 0$.
\end{remark}

\begin{proof}
\par Let $\cS$ be the set of colorings that satisfy conditions \textit{1} through \textit{6} in Lemma \ref{non-crossing matching of case2}, and let $\cS'_u$ be the set of colorings in $\cS$ such that $i$ and $u(i)'$ have different colors for all $i\in [n]$.

\par Consider a $\cm_{I,J}$ term in the sum given in Proposition \ref{sumcm6} (so the coloring $(I, J) \in \cS$). By the remark after Definition \ref{def:cm}, notice that $x_u$ has nonzero coefficient in $\cm_{I,J}$ if and only if the coloring $(I, J) \in \cS_u'$ (in which case the coefficient is $\sgn(u)$.) Thus, we get

\begin{equation}\label{eqn:extract_coeff2}
    f_w (u) = \sgn(w) \sgn(u) |\cS_u'|.
\end{equation}

\par First, note that $\cS_u' = \emptyset$ if some $i \in [1,a+e]$ satisfies $u(i) \in [1,b+f]$. This is because both $i$ and $u(i)$ will be black vertices. Similarly, if some $i \in [a+e+b+c+1, n]$ satisfies $u(i) \in [b+f+a+d+1, n]$, then both $i$ and $u(i)$ will be white vertices. In this case, we find that $f_w (u) = 0.$ Thus, assume neither condition holds. We will now count the elements in $\cS_u'$.

Note that $[a+e+1,a+e+b+c] \cap u^{-1}([b+f+a+d+1,n]) $ must all be black and $[a+e+1,a+e+b+c] \cap u^{-1}([1,b+f])$ must all be white, and the vertices in $[a+e+1, a + e + b+c]$ contains $c$ black and $b$ white vertices. Therefore, we have the freedom of choosing exactly $A$ black and $B$ white vertices from the $A+B$ vertices in $[a + e + 1, a + e + b + c],$ each of which uniquely determines the coloring for $[b + f + 1, b + f + a + d + 1]'$ and thus the entire coloring. As a result, $|\cS_u'| = \binom{A+B}{A}$, and hence by \eqref{eqn:extract_coeff2}, $f_w (u) = \sgn(w)\sgn(u)\binom{A + B}{A}$. 


\end{proof}

Of special interest is the antidiagonal coefficient, which we promote to its own corollary.

\begin{cor}\label{cor:anti_diag}

Let $w$ avoid the patterns $1324$ and $321$ but not $2143$. Then, $w$ falls into one of the two cases of Proposition \ref{2143Patterns}. Define $a, b, c, d \ge 1$ accordingly (we ignore $e, f$). If $w_0$ is the longest word in $\mathfrak{S}_n,$ then $|f_w(w_0)| = \binom{\min(a, c) + \min(b, d)}{\min(b, d)}.$
\end{cor}
\begin{proof}
\par We divide up our work into the two cases, given by Theorems \ref{coefficient of u in w case 2} and  \ref{AntiDiagCoeff}. 
\par For the first case (covered by Theorem \ref{AntiDiagCoeff}), it's not hard to see that $w_0$ lies in the second case in Theorem \ref{AntiDiagCoeff}, meaning that we need to compute the number of elements in $w_0^{-1}([n + 1 - c, n]) \cap [1, a]$ and $w_0^{-1}([1, b]) \cap [n + 1 - d, n].$ However, notice that $w_0$ sends $[1, b]$ to $[n + 1 - b, n],$ meaning that the size of the latter set is $\min(b, d).$ Similarly, we see that the size of the former set is $\min(a, c),$ giving us that by Theorem \ref{AntiDiagCoeff} that $f_w(w_0) = \binom{\min(a, c) + \min(b, d)}{\min(b, d)}.$
\par For the second case (covered by Theorem \ref{coefficient of u in w case 2}), observe that for $w_0,$ $A$ is equal to $c - |[b + f + a  + d + 1, n] \cap [n - a - e - b - c + 1, n - a - e]|.$ But $n - a - e - b - c = d + f,$ meaning that our set has magnitude $|[b + f + a + d + 1, n - a - e]|.$ But observe that $n - a - e = b + f + c + d,$ meaning that this is equal to $\max(c - a, 0).$ Similarly, for $B,$ this is equal to the size of the set $[d + f, b + f] = \max(b - d, 0)$ subtracted from $b.$
\par Therefore, $A = c - \max(c - a, 0) = \min(a, c), B = b - \max(b - d, 0) = \min(b, d),$ whereby Theorem \ref{coefficient of u in w case 2} gives us the desired result.
\par Combining these together gives us the value of the antidiagonal coefficient, as desired.
\end{proof}

\begin{remark}
In particular, since $a, b, c, d \ge 1$, we have $|f_w (w_0)| \ge 2$, so $\imm_w$ cannot be expressed as a \%-immanant. This provides an alternate proof for the 2143-avoiding condition in Theorem \ref{2143ConverseThm}.
\end{remark} 

\subsection{Proof of Theorem \ref{TwoPercentForward}}
Using the results of the previous subsection, we are now able to prove Theorem \ref{TwoPercentForward}, which we restate here.
\begin{thm}\label{TwoPercentForward_restate}
Let $w$ be a $321$-avoiding permutation. The following statements are equivalent:
\begin{enumerate}
    \item The Temperley-Lieb immanant $\imm_w$ is a linear combination of \%-immanants;
    
    \item The signed Temperley-Lieb immanant $\sgn(w) \imm_w$ is a sum of at most two \%-immanants;
    
    \item The permutation $w$ avoids the patterns $1324, 24153, 31524, 231564$, and $312645,$ in addition to avoiding $321.$
\end{enumerate}
\end{thm}
    
    
\begin{proof}
We will prove $\textit{3} \Rightarrow \textit{2} \Rightarrow \textit{1} \Rightarrow \textit{3}$. The implication $\textit{2} \Rightarrow \textit{1}$ is straightforward.

$\textit{3} \Rightarrow \textit{2}$. Assume that $w$ avoids $321, 1324, 24153, 31524.$ If furthermore, $w$ avoids $2143$, then $\imm_w = \sgn(w)\imm_w^\%$ by Theorem \ref{CSBThm}. Thus, we may further assume that $w$ contains $2143$. By Lemma \ref{lem:24153, 31524 pattern}, we have that $w$ has block structure $[2][1][3][5][4]$, as stated in the first case of Proposition \ref{2143Patterns}. Since $w$ furthermore avoids $231564$ and $312645$, we must have $a = 1$ or $c = 1$, and $b = 1$ or $d = 1$. This is because otherwise, the block structure given in the first case of Proposition \ref{2143Patterns} will result in a pattern of $231564$ or $312645$. We will first consider the case when $a = 1$. This case has two subcases. 

\textbf{Case 1.} Assume $a = b = 1$. We use Theorem \ref{AntiDiagCoeff} to prove that $\sgn(w)\imm_w$ is the sum of the following two \% immanants. 
\begin{itemize}
    \item $\imm^\%_1 = \imm_w^{\%}.$
    \item The \%-immanant $\imm^\%_2$ corresponding to the \% shape where we remove the $d \times c$ rectangle in the lower-right corner, and remove the $(n-d) \times 1$ and $1 \times (n-c)$ rectangles in the upper-left corner (with these two rectangles overlapping at the $1 \times 1$ rectangle in the upper-left corner).
\end{itemize}
\par In particular, in the perfect matching of $w$, block $[1,a]$ is matched with block $[(b+1)', (b+a)']$, and block $[1',b']$ is matched with block $[a+1, a+b],$ letting $a=b=1$ gives us that $w(1)=2$ and $w(2)=1$.


\textbf{Case 2.} If $a = d = 1,$ we take the following two \%-immanants:
\begin{itemize}
    \item $\imm^\%_1 = \imm_w^{\%}.$
    \item The \%-immanant $\imm^\%_2$ corresponding to the shape where we remove the $1 \times (n-b)$ rectangle in the lower-right corner, and remove the $(n-c) \times 1$ rectangle in the upper-left corner.
\end{itemize} 

To show both of these, we compare the coefficients of $x_u$ for each $u \in \fkS_n.$ First, suppose that $u([1, a])$ is not disjoint from $[1, b],$ or $u([n - d + 1, n])$ is not disjoint from $[n-c + 1, n],$ so we know that $f_w(u) = 0$ by Theorem \ref{AntiDiagCoeff}.
\par Notice that these conditions in particular imply that the coefficient of $x_u$ in $\imm_w^{\%}$ is zero for both cases. This also holds for $\imm_2^{\%}$ as well, since the rectangles that we remove in the latter contain the rectangles we remove in the former. So the coefficient of $x_u$ in $\imm_1^{\%} + \imm_2^{\%}$ is zero as well.

\par Now, consider $u \in \mathfrak{S}_n$ where $f_w(u) \neq 0$. Then, $f_w(u) =$ $A+B \choose A$ as in Theorem \ref{AntiDiagCoeff}. Recall that for $w$ to avoid the above patterns, $w$ fall into one of the four cases: $a = 1, b = 1$, or $a = 1, d = 1$, or $c = 1, b = 1$, or $c = 1, d = 1$. Given these constraints, for $A+B \choose A$ to be nonzero, we can have either $B \choose 0$ $= 1$, or $2 \choose 1$ $= 2$. 
\par Suppose that this coefficient is $1,$ so either $A = 0$ or $B = 0.$ Then, when $a = b = 1,$ note that $A = 0$ iff $u(1) \not \in [n-c+1, n],$ or that it doesn't fit in the second pattern, and $B = 0$ iff $u^{-1}(1) \not \in [n-d+1, n],$ or that it doesn't fit in the second pattern. Meanwhile, for $a = d = 1,$ we have $A = 0$ iff $u(1) \not \in [n - c + 1, n]$ (so doesn't fit in the second pattern), and $B = 0$ iff $u^{-1}([1, b])$ is not equal to $n,$ or that it doesn't fit in the second pattern.
\par Otherwise, the coefficient is ${2 \choose 1} = 2$ if and only if $A \neq 0$ and $B \neq 0.$ But then we have that $u(1) \in [n - c + 1, n]$ and $u([n-d + 1, n]) \cap [1, b]$ contains an element. In the case for $a = b = 1,$ we have that $u(1) \not \in [1, n - c]$ and $u([1, n-d]) \neq 1$ (since $u([n-d+1, n])$ contains $1$), and the condition that $f_w(u) \neq 0$ means that $u([n-d+1, n])$ is disjoint from $[n-c+1, n],$ so $u$ fits in both \%-immanant patterns. For the case $a = d = 1,$ again $u(1)$ does not lie in $[1, n-c]$ and $u(n)$ is not in $[b + 1, n],$ so $u$ fits in both $\%$-immanants. This proves our result for $a = 1$.

Otherwise, if $c = 1$, then notice that $w'=w_0 w^{-1} w_0$ will also have block structure $[2][1][3][5][4]$ and the corresponding $a$-value for $w'$ is $1$. This is because taking the inverse is the same as reflecting the matching diagram across the perpendicular bisector of $1$ and $1'$, while taking conjugation by $w_0$ is the same as reflecting the matching diagram across the perpendicular bisector of $1$ and $n$. By the discussion of our previous case, $\sgn(w') \imm_{w'}$ is a sum of two \%-immanants $\imm^\%_1 + \imm^\%_2$. And by Lemma \ref{lem:symmetry}, $\sgn(w) \imm_w = \imm^{*\%}_1 + \imm^{*\%}_2$, where $\imm^{*\%}_i$ is the $\%$-immanant with a zero in $(i, j)$ if and only if $\imm^\%_i$ has a zero in $(n+1-j, n+1-i)$.


$\textit{1} \Rightarrow \textit{3}$. We will show the contrapositive of this statement. Suppose $w$ doesn't avoid one of $1324, 24153, 31524, 231564, 312645$; we would like to show $\imm_w$ is not a linear combination of $\%$-immanants. If $w$ contains $1324$, then we are done by Theorem \ref{1324Thm}, so we can assume $w$ avoids $1324$.

By the symmetry lemma \ref{lem: simplify cases}, we know that $\imm_w$ is a linear combination of $\%$-immanants iff $\imm_{w^{-1}}$ is. Thus, by replacing $w$ with $w^{-1}$ and using Lemma \ref{lem: restriction inverses} if necessary, we may assume without loss of generality that $w$ contains $24153$ or $231564$ and avoids $1324$. In particular, $w$ contains $2143$.

The basic idea is to find two $1324$-adjacent permutations $v, v'$ such that $f_w (v) \neq -f_w (v')$, and use Theorem \ref{thm:classifying space of percent} to conclude that $\imm_w$ is not a linear combination of $\%$-immanants. We divide into cases per Proposition \ref{2143Patterns}. 
\par \textbf{Case 1: $w$ has block form $[2][1][3][5][4]$.} Define $a, b, c, d, e$ as in the first case of Proposition \ref{2143Patterns}. Then by inspection of the block form, we see that $w$ must avoid $24153$, so $w$ contains $231564.$ Further inspection of the block form also yields that $\min(a, c) \geq 2.$
\par Consider the permutation $v$ given by $v(i) = i + (n - a - c)$ for $i \leq a + c,$ and $v(i) = n + 1 - i$ for all other values of $i.$ We first note that $v([1, a]) = [n - a - c+1, n-c],$ which is disjoint from $[1, b]$ and $[n + 1 - c, n],$ and that $v([n + 1 - d, n]) = [1, d],$ which is disjoint from $[n + 1 - c, n]$ as $\max(b, d) \leq n - a - c - e \leq n - c.$ Therefore, by Theorem \ref{AntiDiagCoeff}, we have that $f_w(v) = \sgn(w)\sgn(v),$ since $[1,a] \cap u^{-1}([n+1 - c,n])$ is empty and $\binom{N}{0} = 1$ for any $N$.
\par Now, consider the permutation $vs_a.$ Repeating the computations, observe that $vs_a([1, a])$ is $[n - a - c, n - c - 1] \cup \{n - c + 1\},$ which is disjoint from $[1, b]$ and intersects $[n+1-c, n]$ in one element, and $vs_a([n + 1 - d, n]) = [1, d],$ which is disjoint from $[n + 1 - c, n].$ Again, by Theorem \ref{AntiDiagCoeff}, we find that $f_w(vs_a) = \sgn(w)\sgn(vs_a) \binom{\min(b, d)+1}{1} = \sgn(w)\sgn(vs_a) (\min(b, d)+1),$ since $vs_a([n + 1 - d, n])$ contains $\min(b, d)$ elements that map to elements in $[1, b].$
\par Thus, $f_w (v) = \pm 1$ while $f_w (vs_a) = \mp (\min(b,d) + 1)$. But $v, vs_{a+e}$ are $1324$-adjacent, by considering the inputs $a-1, a, a+1, a+2$ (it is here that the assumption $\min(a, c) \ge 2$ is used). Thus, by Theorem \ref{thm:classifying space of percent}, $\imm_w$ cannot be written as a sum of \%-immanants.
\par \textbf{Case 2: $w$ has block form $[3][5][1][6][2][4]$.} Define $a, b, c, d, e, f$ as in the second case of Proposition \ref{2143Patterns}, with $\max(e, f) > 0$. By replacing $w$ with $w^{-1}$ if needed, we can without loss of generality assume $e \geq 1,$ which means that our permutation contains the pattern $24153.$
\par In this case, consider the permutation given by $v(i) = i + n - a - 2e - c$ for $e \le i \le a + e + c + 1$, and $w(i) = n + 1 - i$ otherwise. Notice that $v([a+e+1, a+e+c+1]) = [n - e - c + 1, n - e + 1] \subset [b + f + a + d + 1, n]$ since $n = a + b + c + d + e + f.$ Therefore, by Theorem \ref{coefficient of u in w case 2}, we have that $f_w(v) = 0,$ since $A = c - (c + 1) < 0.$
\par Now, consider the permutation $vs_{a + e}.$ Observe first that $vs_{a+e}([1, a+e]) = [n - e + 1, n] \cup [n - a - e - c, n - e - c - 1] \cup \{ n - e - c + 1\};$ but $n - a - e - c  = b + f + d \geq b + f + 1.$ Similarly, note that $vs_{a+e}([a + e + b + c + 1, n]) = [1, f + d];$ again note that $f + d < b + f + a + d + 1.$
\par Finally, observe that $vs_{a+e}([a + e + 1, a + e + b + c])$ equals the union of $v([a+e+2, a+e+b+c])$ and $\{v(a+e)\}.$ This in turn equals $[n - e - c + 2, n - e + 1] \cup [n - a - e - b - c + 1 , n - a - e - c - 1] \cup \{n - e - c\} = [b + f + a + d + 2, n - e + 1] \cup [d + f + 1 , b + d + f - 1] \cup \{b + f + a + d\}.$ This contains $c$ elements in $[b + f + a + d + 1, n]$ and $\max(b - d, 0)$ elements of $[1, b + f]$, so by Theorem \ref{coefficient of u in w case 2} we have that $f_w(vs_{a+e}) = \sgn(vs_{a+e})\sgn(w).$
\par Thus, $f_w(vs_{a+e})$ is equal to $\pm 1$ while $f_w (v) = 0$. But $v, vs_{a+e}$ are $1324$-adjacent, by considering the inputs $a + e - 1, a + e, a + e + 1, a + e + 2.$ Thus, by Theorem \ref{thm:classifying space of percent}, $\imm_w$ is not a linear combination of \%-immanants.
\end{proof}


\subsection{Proofs of Lemmas in Section 5}
Here we present the proofs for the lemmas introduced in section 5.

\begin{replemma}{lem:24153, 31524 pattern}
Suppose that $w$ is a permutation avoiding $1324, 321$ that contains the pattern $2143,$ and define $a', b', c', d'$ as in Proposition \ref{2143Patterns}. Then, suppose that $a' + b' + c' + d' > n.$ Then, $w$ either contains the pattern $24153$ or $31524.$
\end{replemma}
\begin{proof}
We must be in Case 2 of Proposition \ref{2143Patterns}. If $e \ge 1$, then $w$ contains a $24153$ pattern. If $f \ge 1$, then $w$ contains a $31524$ pattern.
\end{proof}

\begin{replemma}{general non-crossing matching}
Let $a, b, c, d, e \ge 0$ be integers with $n = a + b + c + d + e$. Consider vertices labelled $1, 2, \ldots, 2n.$ Then there exists a unique coloring of these vertices with colors white and black and a non-crossing matching compatible with it such that the following conditions hold. 
\begin{enumerate}
    \item $[1, b+c+e]$ are colored black,
    
    \item in the interval $[b+c+e+1, a+2b+c+e]$, there are exactly $a$ black and $b$ white vertices, and there are no pairings between two vertices in this interval (which we will refer to as an ``internal pairing''),
    
    \item $[a+2b+c+e+1, a+b+e+n]$ are colored white,
    
    \item in the interval $[a+b+e+n + 1, 2n]$, there are exactly $d$ black and $c$ white vertices, and there are no internal pairings in this interval.
\end{enumerate}
\end{replemma}

\begin{proof}
The proof will proceed similarly to the proof of Lemma \ref{lem:simple-no-internal-pair}. We will induct on $c + d$. Our base case is $c = d = 0.$ In this case, we wish to show there is a unique coloring and non-crossing matching compatible with the coloring such that $[1, b + e]$ are colored black, $[b + e + 1, a + 2b + e]$ have no internal pairings and have $a$ black, $b$ white vertices, and $[a + 2b + e, 2n]$ are colored white. But this is just Lemma \ref{lem:simple-no-internal-pair} with the values $a+b, b+e, a+e,$ as $(a + b) + (b + e) + (a + e) = 2n,$ and then shifted by $a + b.$

For the inductive step, assume that we have shown that the lemma is true for $c + d = k-1$. Suppose we are given nonnegative integers $a, b, c, d, e$ with $c + d = k$. We can assume without loss of generality that $c \ge d.$ Otherwise, relabel the vertices such that $x$ is relabelled as $2n + 1 - c - d - x$ for $x \in [1, a + b + e + n]$ and $4n + 1 - c - d - x$ for $x \in [a + b + e + n + 1, 2n],$ and swap the colors black and white. This has the effect of swapping the variables $a$ with $b$ and $c$ with $d$. 

\par From our assumptions that $c \ge d$ and $c + d \ge 1$, we must have $c \ge 1$. 

\par Now, suppose a coloring and a non-crossing matching compatible with the coloring satisfying the properties in the statement of the lemma exist. We claim that $2n$ must be paired with $1$. If not, then $2n$ is paired with $x$ for some $2 \le x \le a + b + e + n = 2n-c-d$. Then among the vertices $[1, x-1]$, there must be equal number of white and black vertices. Otherwise, one of these vertices must be paired with a vertex with index at least $x,$ contradicting the definition of a non-crossing matching. 

\par We now show that this can't happen. If $1 < x \le b+c+e + 1$, then $[1, x-1]$ consists of only black vertices. Next, if $b +c + e + 1 < x \le a+2b+c+e + 1$, then $[1, x-1]$ has at least $b+c+e$ black vertices and at most $b$ white vertices by condition \textit{2}. This is a contradiction since $c \ge 1$, which means $b + c + e > b$. Finally, if $a+2b+c+e + 1 < x \le 2n-c-d$, then $[1, x-1]$ has at least $a+b+c+e$ black vertices and at most $b + (a+d+e-1)$ white vertices by conditions \textit{2} and \textit{3}, which is a contradiction since $c > d-1$. Thus, we must have $2n$ is paired with $1$. In particular, we also know that $2n$ is colored white, since $1$ is colored black.

Now, delete vertices $1$ and $2n,$ and relabel the vertices $[2, 2n-1]$ by reducing their index by $1.$ The remaining configuration must satisfy the following properties.
\begin{enumerate}
    \item $[1, b + c + e - 1]$ are colored black,
    \item in the interval $[b + c + e, a + 2b + c + e - 1],$ there are $a$ black and $b$ white vertices, and there are no internal pairings,
    \item $[a + 2b + c + e, a + b + e + n - 1]$ are colored white,
    \item in the interval $[a + b + e + n, 2(n-2)],$ there are exactly $d$ black and $c - 1$ white vertices, and there are no internal pairing in this interval.
\end{enumerate}
By the inductive hypothesis, this choice of non-crossing matching with compatible coloring on $2n - 2$ vertices is unique. However, this means that the choice of non-crossing matching and compatible coloring on the $2n$ vertices is also unique, since $1$ must be colored black, $2n$ must be colored white, and $1, 2n$ must be paired. This shows that if a coloring and compatible non-crossing matching satisfying the properties in the lemma statement exist, then they must be unique.
\par To see that one such choice of coloring and non-crossing matching exists, take the unique coloring and non-crossing matching on the vertices $[1, 2n - 2]$ (which exists by the inductive hypothesis) satisfying the four properties listed in the previous paragraph. Then, shift the labels of the vertices by $1,$ and add the vertices $1, 2n.$ From here, color $1$ black, $2n$ white, and pair $1$ and $2n.$ This is a coloring and compatible non-crossing matching that satisfies the conditions in the lemma. This finishes the inductive step and hence proves the claim. \qedhere 

\end{proof}

\begin{replemma}{non-crossing matching of case1}
Let $a, b, c, d, e$ be nonnegative integers, so $a, b, c, d \geq 1$ and $a+b+c+d+e=n$. Then, there is a unique non-crossing matching and a coloring compatible with it, such that the coloring satisfies
\begin{enumerate}
    \item $i$ is black for $i\in [a+1,n-d]$
    \item $i'$ is white for $i\in [b+1, n-c]$
    \item There are exactly $a$ black vertices and $b$ white vertices in $[1,a]\cup [1,b]'$
    \item There are exactly $d$ black vertices and $c$ white vertices in $[n-d+1,n]\cup [n-c+1,n]'$ 
    \item There are no pairings between two vertices in $[1,a]\cup [1,b]'$ (which we will refer to as an ``internal pairing")
    \item There are no internal pairings in $[n-d+1,n]\cup [n-c+1,n]'$
\end{enumerate}
\end{replemma}

\begin{proof}
Relabel the vertices such that $i$ is relabelled as $n - d + 1 - i$ for $i \in [1, n-d],$ relabelled as $3n - d  + 1 - i$ for $i \in [n - d + 1, n],$ and $i'$ is relabelled as $i + n - d$ for $i \in [1, n].$ Then, our conditions can be restated as requiring $[1, b + c + e]$ to be black, $[b + n - d + 1, 2n - c - d] = [a + 2b + c + e + 1, a + b + e + n]$ to be white, with $a$ black and $b$ white vertices in $[b + c + e + 1, a + 2b + c + e],$ and $d$ black and $c$ white vertices in $[2n - c - d + 1,2n] = [a + b + e + n + 1, 2n],$ with no internal pairings in the last two intervals. By Lemma \ref{general non-crossing matching}, there exists a unique coloring and non-crossing matching satisfying these conditions, as desired.
\end{proof}

\begin{replemma}{non-crossing matching of case2}
Let $a, b, c, d, e, f$ be nonnegative integers where $a, b, c, d, \max(e, f) \geq 1,$ and $a+b+c+d+e+f=n$. Then, there is a unique non-crossing matching and a coloring compatible with it, such that the coloring satisfies
\begin{enumerate}
    \item $i$ is black for $i\in [1,a+e]$
    \item $i$ is white for $i\in [a+e+b + c+1, n]$
    \item $i'$ is black for $i\in [1,b+f]$
    \item $i'$ is white for $i\in [b+f+a+d+1,n]$
    \item There are exactly $c$ black vertices and $b$ white vertices in $[a+e+1,a+e+b+c]$
    \item There are exactly $d$ black vertices and $a$ white vertices in $[b+f+1,b+f+a+d]'$ 
    \item There are no pairings between two vertices in $[a+e+1,a+e+b+c]$ (which we will refer to as an ``internal pairing")
    \item There are no internal pairings in $[b+f+1,b+f+a+d]'$
\end{enumerate}
\end{replemma}
\begin{proof}
Let $\Tilde{a}=d,\Tilde{b}=a,\Tilde{c}=b,\Tilde{d}=c$, and $\Tilde{e}=e+f$. Adopt the setting of Lemma \ref{general non-crossing matching} where we take the five constants to be $\Tilde{a},\Tilde{b},\Tilde{c},\Tilde{d}$, and $\Tilde{e}$.

Relabel the vertices such that $i$ is labelled as $a + e + 1 - i$ for $i \in [1, a + e],$ labelled $2n + a + e + 1 - i$ for $i \in [a + e + 1, n],$ and $i'$ is labelled as $a + e + i$ for $i \in [1, n].$ Then, notice that we are looking for a coloring and non-crossing matching compatible with it such that $[1, a + e + b + f] = [1, \Tilde{b} + \Tilde{c} + \Tilde{e}]$ are black, $[2a + b + d + e + f + 1, n + a + e + d + f] = [\Tilde{a} + 2\Tilde{b} + \Tilde{c} + \Tilde{e} + 1, n + \Tilde{a} + \Tilde{b} + \Tilde{e}]$ are white, and we have $a$ white and $d$ black vertices in $[a +e + b + f + 1, 2a + b + d + e + f] = [\Tilde{b} + \Tilde{c} + \Tilde{e} + 1, \Tilde{a} + 2\Tilde{b} + \Tilde{c} + \Tilde{e}],$ and $c$ black and $b$ white vertices in $[n + a + e + d + f + 1, 2n] = [n + \Tilde{a} + \Tilde{b} + \Tilde{e} + 1,2n],$ with no internal pairings amongst the last two intervals. By Lemma \ref{general non-crossing matching}, there is a unique coloring and non-crossing matching compatible with it, as desired.
\end{proof}

\begin{replemma}{ncm equals w case1}
Let $w \in \fkS_n$ have block structure $[2][1][3][5][4]$ with block lengths $a,b,e,c,d$, as stated in the first case of Proposition \ref{2143Patterns}. The non-crossing matching of $w$ is exactly the non-crossing matching in Lemma \ref{non-crossing matching of case1}.
\end{replemma}
\begin{proof}
\par The first part of the proof is to determine $\ncm(w)$. This can be done laboriously using Proposition \ref{prop:ncm}, but for our purposes, it suffices to deduce some structural properties of $\ncm(w)$ that can be directly gleaned from Lemma \ref{lem:matching_compatibility} and Proposition \ref{prop:ncm}(a).

\par Construct the coloring of $\ncm(w)$ in Lemma \ref{lem:matching_compatibility}, with fixed points colored arbitrarily. Then, using the one-line notation for $w$ in Proposition \ref{2143Patterns} case 1, we deduce that the vertices in $B_1 = [1,a] \cup [1,b]'$ are all colored black, while the vertices in $W_1 = [a+1, a+b] \cup [b+1, a+b]'$ are all colored white. Thus, since $\ncm(w)$ is consistent with this coloring, there are no pairings between two vertices in $B_1$. There are also no pairings between two vertices in $W_1$. Finally, by Lemma \ref{lem:matching_compatibility}, a white vertex $i$ or $i'$ in $W_1$ must be paired with some black vertex $j$ or $j'$ with $j \le i \le a+b$. The only black vertices in $[1, a+b] \cup [1, a+b]'$ belong to $B_1$, so each vertex in $W_1$ is paired with some vertex in $B_1$. Since $|W_1| = |B_1| = a+b$, this map is bijective.
\par A similar argument shows that each vertex in $W_2 = [n-d+1,n] \cup [n-c+1,n]'$ is paired with some vertex in $B_2 = [n-c-d, n-d] \cup [n-c-d, n-c]'$ (in particular, there are no internal pairings between vertices in $W_2$). Finally, note that for $i \in F := [a + b + 1, n - c - d]$, we have $w(i) = i$ by Proposition \ref{2143Patterns}, so $i$ is paired with $i'$ in $\ncm(w)$ by Proposition \ref{prop:ncm}(a). 

\par Using the properties of $\ncm(w)$ deduced from above, we will create a new coloring compatible with $\ncm(w)$ (\textbf{not} based on Lemma \ref{lem:matching_compatibility}) that satisfies the conditions of Lemma \ref{non-crossing matching of case1}. First, conditions \textit{5} and \textit{6} of Lemma \ref{non-crossing matching of case1} are satisfied from our above work. Next, we follow conditions \textit{1} and \textit{2} of Lemma \ref{non-crossing matching of case1} and color $i$ black for $i \in [a+1, n-d]$ and $i'$ white for $i \in [b+1, n-c]$. This determines the colors of vertices in $W_1, F, F'$, and $B_2$. Since $\ncm$ pairs vertices in $W_1$ with $B_1$, $W_2$ with $B_2$, and $F$ with $F'$, our coloring is consistent so far. Furthermore, we can easily extend our coloring to all of $[1, n] \cup [1, n]'$: for each black vertex $v_1 \in B_1$ paired to some white vertex $v_2 \in W_1$, we assign $v_1$ the opposite color of $v_2$, and likewise for vertices in $W_2$.
\par Now, we verify the remaining conditions of Lemma \ref{non-crossing matching of case1}. Condition \textit{3} is satisfied because there are exactly $a$ white vertices and $b$ black vertices in $W_1$, which will thus pair with $a$ black vertices and $b$ white vertices in $B_1 = [1,a] \cup [1,b]'$. Similarly, condition \textit{4} of Lemma \ref{non-crossing matching of case1} is also satisfied.

\par Thus, it follows that $\ncm(w)$ is the unique matching described in Lemma \ref{non-crossing matching of case1}, as desired.
\end{proof}

\begin{replemma}{ncm equals w case2}
Let $w \in \fkS_n$ have block structure $[3][5][1][6][2][4]$ with block lengths $a,e,b,c,f,d$, as stated in the second case of Proposition \ref{2143Patterns}. The non-crossing matching of $w$ is exactly the non-crossing matching in Lemma \ref{non-crossing matching of case2}.
\end{replemma}
\begin{proof}
Consider the following coloring.
\begin{itemize}
    \item $[1, a+e], [a+e+b+1,a+e+b+c], [1, b+f]', [b+f+a+1, b+f+a+d]'$ are black,
    
    \item $[a+e+1, a+e+b], [a+e+b+c+1, n], [b+f+1, b+f+a]', [b+f+a+d+1, n]'$ are white.
\end{itemize}
Using the one-line notation for $w$ in Proposition \ref{2143Patterns} case 2, we check that this coloring satisfies the conditions in Lemma \ref{lem:matching_compatibility}, and so this coloring is compatible with $\ncm(w)$.
\par Furthermore, note that this coloring satisfies conditions \textit{1-6} in Lemma \ref{non-crossing matching of case2}. Next, we check conditions \textit{7-8}. From Lemma \ref{lem:matching_compatibility}, we know that any black vertex has a smaller or equal label than the white vertex it is paired with in $\ncm(w)$. Thus, $\ncm(w)$ cannot have pairings between any two vertices in $[a + e + 1, a + e + b + c],$ since $[a + e + 1, a + e + b]$ are colored white and $[a + e + b + 1, a + e + b + c]$ are colored black. The same holds for the vertices among $[b + f + 1, b + f + a + d]'.$ Thus, conditions \textit{1-8} in Lemma \ref{non-crossing matching of case2} are satisfied, so $\ncm(w)$ is the unique non-crossing matching given in Lemma \ref{non-crossing matching of case2}. \qedhere
\end{proof}

\section*{Acknowledgements}
This project was partially supported by RTG grant NSF/DMS-1148634, DMS-1949896, and the Office of Undergraduate Research at Washington University in St. Louis. It was supervised as part of the University of Minnesota School of Mathematics Summer 2021 REU program. 
The authors would like to thank Professor Pavlo Pylyavskyy for introducing the problem and offering helpful directions for research and Sylvester Zhang for their mentorship and helpful comments on the paper. In addition, the authors would like to thank Swapnil Garg and Brian Sun for their algorithmic and coding support for this project.


\bibliographystyle{amsplain}
\bibliography{paper1.bib}

\end{document}